\RequirePackage{fix-cm}
\documentclass[smallextended]{svjour3}       
\smartqed  
\usepackage{graphicx}
\usepackage[misc]{ifsym}
\usepackage{multirow}
\usepackage{amsmath,amssymb}
\usepackage{float}
\usepackage{url}
\usepackage[ruled]{algorithm2e}
\usepackage[left=1.3in, right=1.3in]{geometry}
\usepackage{caption}
\usepackage{algpseudocode}
\usepackage[colorlinks,linkcolor=blue,citecolor=blue]{hyperref}
\if@onethmnum
\newtheorem{assumption}[theorem]{Assumption}
\else
\newtheorem{assumption}[theorem]{Assumption}
\fi
\begin{document}

\title{A FE-inexact heterogeneous ADMM for Elliptic Optimal Control Problems with {$L^1$}-Control Cost}

\author{Xiaoliang Song$\mathbf{^1}$ \and  Bo Yu$\mathbf{^{1}}$ \and Yiyang Wang$\mathbf{^1}$\and \and Xuping Zhang Wang$\mathbf{^1}$}


\institute{Xiaoliang Song\at
           \email{songxiaoliang@mail.dlut.edu.cn}
           \and
            Bo Yu \at
           \email{yubo@dlut.edu.cn}
           \and
            Yiyang Wang \at
           \email{yywerica@gmail.com}
            \and
           Xuping Zhang \at
           \email{zhangxp@dlut.edu.cn}
           \and
           \at
           $\mathbf{^1}$ School of Mathematical Sciences,
Dalian University of Technology, Dalian, Liaoning, 116025, China
}



\maketitle

\begin{abstract}
 Elliptic PDE-constrained optimal control problems with $L^1$-control cost ($L^1$-EOCP) are considered. To solve $L^1$-EOCP, the primal-dual active set (PDAS) method, which is a special semismooth Newton (SSN) method, used to be a priority. However, in general solving Newton equations is expensive. Motivated by the success of alternating direction method of multipliers (ADMM), we consider extending the ADMM to $L^1$-EOCP. To discretize $L^1$-EOCP, the piecewise linear finite element (FE) is considered. However, different from the finite dimensional $l^1$-norm, the discretized $L^1$-norm does not have a decoupled form. To overcome this difficulty, an effective approach is utilizing nodal quadrature formulas to approximately discretize the $L^1$-norm and $L^2$-norm. It is proved that these approximation steps will not change the order of error estimates. To solve the discretized problem, an inexact heterogeneous ADMM (ihADMM) is proposed. Different from the classical ADMM, the ihADMM adopts two different weighted inner product to define the augmented Lagrangian function in two subproblems, respectively. Benefiting from such different weighted techniques, two subproblems of ihADMM can be efficiently implemented. Furthermore, theoretical results on the global convergence as well as the iteration complexity results $o(1/k)$ for ihADMM are given. In order to obtain more accurate solution, a two-phase strategy is also presented, in which the primal-dual active set (PDAS) method is used as a postprocessor of the ihADMM. Numerical results not only confirm error estimates, but also show that the ihADMM and the two-phase strategy are highly efficient.
\end{abstract}

\section{Introduction}
\label{intro}
In this paper, we study the following non-differentiable optimal control problem with \textsl{$L^1$}-control cost, which is known to lead to sparse controls:
\begin{equation}\label{eqn:orginal problems}
           \qquad \left\{ \begin{aligned}
\min \limits_{(y,u)\in Y\times U}^{}\ \ J(y,u)&=\frac{1}{2}\|y-y_d\|_{L^2(\Omega)}^{2}+\frac{\alpha}{2}\|u\|_{L^2(\Omega)}^{2}+\beta\|u\|_{L^1(\Omega)} \\
{\rm s.t.}\qquad\qquad Ay&=u+y_c\ \ \mathrm{in}\  \Omega, \\
y &=0\quad  \mathrm{on}\ \partial\Omega,\\
u &\in U_{ad}=\{v(x)|a\leq v(x)\leq b, {\rm a.e }\  \mathrm{on}\ \Omega\}\subseteq U,
                          \end{aligned} \right.\tag{$\mathrm{P}$}
 \end{equation}
where $Y:=H_0^1(\Omega)$, $U:=L^2(\Omega)$, $\Omega\subseteq \mathbb{R}^n$, $n=2$ or $3$, is a convex, open and bounded domain with $C^{1,1}$- or polygonal boundary $\Gamma$; $y_c$, $y_d \in L^2(\Omega)$ and parameters $-\infty<a<0<b<+\infty$, $\alpha$, $\beta>0$. Moreover, the operator $A$ is a second-order linear elliptic differential operator. Such the optimal control problem (\ref{eqn:orginal problems}) plays an important role in the placement of control devices \cite{Stadler}. In some cases, it is difficult or undesirable to place control devices all over the control domain and hope to localize controllers in small and most effective regions.

For the study of optimal control problems with sparsity promoting terms, as far as we know, the first paper devoted to this study is published by Stadler \cite{Stadler}, in which structural properties of the control variables were analyzed and two Newton-typed algorithms (including the semismooth Newton algorithm and the primal-dual active set method) were proposed in the case of the linear-quadratic elliptic optimal control problem. In 2011, a priori and a posteriori error estimates were first given by Wachsmuth and Wachsmuth in \cite{WaWa} for piecewise linear control discretizations, in which the convergence rate is obtained to be of order $\mathcal{O}(h)$ under the $L^2$ norm.
In a sequence of papers \cite{CaHerWa1,CaHerWaoptimal}, for the non-convex case governed by a semilinear elliptic equation, Casas et al. proved second-order necessary and sufficient optimality conditions. Using the second-order sufficient optimality conditions, the authors provided an error estimates of order $h$ w.r.t. the $L^\infty$ norm for three different choices of the control discretization (including the piecewise constant, piecewise linear control discretization and the variational control discretization).

Next, let us mention some existing numerical methods for solving problem (\ref{eqn:orginal problems}). Since problem (\ref{eqn:orginal problems}) is nonsmooth, thus applying semismooth Newton methods is used to be a priority. A special semismooth Newton method with the active set strategy, called the primal-dual active set (PDAS) method is introduced in \cite{BeItKu} for control constrained elliptic optimal control problems. It is proved to have the locally superlinear convergence (see \cite{Ulbrich2} for more details). Furthermore, mesh-independence results for semismooth Newton methods were established in \cite{meshindependent}. However, in general, it is expensive in solving Newton equations, especially when the discretization is in a fine level.

Recently, for the finite dimensional large scale optimization problem, some efficient first-order algorithms, such as iterative shrinkage/soft thresholding algorithms (ISTA) \cite{Blumen}, accelerated proximal gradient (APG)-based method \cite{inexactAPG,Beck,inexactABCD}, ADMM \cite{Boyd,SunToh1,SunToh2,Fazel}, etc., have become the state of the art algorithms. Thanks to the iteration complexity $O(1/k^2)$, a fast inexact proximal (FIP) method in function space, which is actually the APG method, was proposed to solve the problem (\ref{eqn:orginal problems}) in \cite{FIP}. As we know, the efficiency of the FIP method depends on how close the step-length is to the Lipschitz constant. However, in general, choosing an appropriate step-length is difficult since the Lipschitz constant is usually not available analytically. 

In this paper, we will mainly focus on the ADMM algorithm. The classical ADMM was originally proposed by Glowinski and Marroco \cite{GlowMarro} and Gabay and Mercier \cite{Gabay}, and it has found lots of efficient applications in a broad spectrum of areas. In particular, we refer to \cite{Boyd} for a review of the applications of ADMM in the areas of distributed optimization and statistical learning.

Motivated by the success of the finite dimensional ADMM algorithm, it is reasonable to consider extending the ADMM to infinite dimensional optimal control problems, as well as the corresponding discretized problems. In 2016, the authors \cite{splitbregmanforOCP} adapted the split Bregman method (equivalent to the classical ADMM) to handle PDE-constrained optimization problems with total variation regularization. However, for the discretized problem,
the authors did not take advantage of the inherent structure of problem and still used the classical ADMM to solve it.

In this paper, making full use of inherent structure of problem, we aim to design an appropriate ADMM-type algorithm to solve problem (\ref{eqn:orginal problems}). In order to employ the ADMM algorithm and obtain a separable 
by adding an artificial variable $z$, we can separate the smooth and nonsmooth terms and equivalently reformulate problem (\ref{eqn:orginal problems}) as:
\begin{equation}\label{eqn:modified problems}
 \left\{ \begin{aligned}
\min \limits_{(y,u,z)\in Y\times U\times U}^{}\ \ J(y,u,z)&=\frac{1}{2}\|y-y_d\|_{L^2(\Omega)}^{2}+\frac{\alpha}{4}\|u\|_{L^2(\Omega)}^{2}+\frac{\alpha}{4}\|z\|_{L^2(\Omega)}^{2}+\beta\|z\|_{L^1(\Omega)} \\
{\rm s.t.}\qquad\qquad\qquad Ay&=u+y_c\ \ \mathrm{in}\  \Omega, \\
y&=0\quad  \mathrm{on}\ \partial\Omega,\\
u&=z,\\
z&\in U_{ad}=\{v(x)|a\leq v(x)\leq b, {\rm a.e }\  \mathrm{on}\ \Omega\}\subseteq U.
                          \end{aligned} \right.\tag{$\mathrm{\widetilde{P}}$}
 \end{equation}
An attractive feature of problem (\ref{eqn:modified problems}) is that the objective function with respect to each variable is strongly convex, which ensures the existence and uniqueness of the optimal solution. Moreover, in many algorithms, strong convexity is a boon to good convergence and makes possible more convenient stopping criteria.

Then an inexact ADMM in function space is developed for (\ref{eqn:modified problems}). 
Focusing on the inherent structure of the ADMM in function space is worthwhile for us to propose an appropriate discretization scheme and give a suitable algorithm to solve the corresponding discretized problem. As will be mentioned in the Section \ref{sec:2}, since each subproblem of the inexact ADMM algorithm for (\ref{eqn:modified problems}) has a well-structure, it can be efficiently solved. Thus, it will be a crucial point in the numerical analysis to construct similar structures for the discretized problem.

To discretize problem (\ref{eqn:modified problems}), we consider using the piecewise linear finite element to discretize the state variable $y$, the control variable $u$ and the artificial variable $z$. However, the resulting discretized problem is not in a decoupled form as the the finite dimensional $l^1$-regularization optimization problem usually does, since the discretized $L^1$-norm does not have a decoupled form:
\begin{equation*}
  \|z_h\|_{L^{1}(\Omega_h)}=\int_{\Omega_h}\left|\sum_{i=1}^{n}z_i\phi_i(x)\right|\mathrm{d}x.
\end{equation*}
Thus, we employ the following nodal quadrature formulas to approximately discretize the $L^1$-norm and we have
\begin{equation*}
  \|z_h\|_{L^{1}_h(\Omega_h)}=\sum_{i=1}^{n}|z_i|\int_{\Omega_h}\phi_i(x)\mathrm{d}x.
\end{equation*}
which has introduced in \cite{WaWa}.
Moreover, in order to obtain a closed form solution for the subproblem of $z$, an similar quadrature formulae is also used to discretize the squared $L^2$-norm:
\begin{equation}\label{equ:approxL2norm}
   \|z_h\|_{L^{2}_h(\Omega_h)}^2={\sum_{i=1}^{n}(z_i)^2}\int_{\Omega_h}\phi_i(x)\mathrm{d}x.
\end{equation}
For the new finite element discretization scheme, we establish a priori finite element error estimate w.r.t. the $L^2$ norm, i.e.
$\|u-u_h\|_{L^2(\Omega)}\leq C(\alpha^{-1}h+\alpha^{-\frac{3}{2}}h^2)$,
which is same to the result shown in \cite{WaWa}.

To solve (\ref{equ:approx discretized matrix-vector form}), i.e., the discrete version of (\ref{eqn:modified problems}), we consider using the ADMM-type algorithm. However, when the classical ADMM is directly used to solve (\ref{equ:approx discretized matrix-vector form}), there is no well-structure as in continuous case and the corresponding subproblems can not be efficiently solved.
Thus, making use of the inherent structure of (\ref{equ:approx discretized matrix-vector form}), an heterogeneous ADMM is proposed. Meanwhile, sometimes it is unnecessary to exactly compute the solution of each subproblem even if it is doable, especially at the early stage of the whole process. For example, if a subproblem is equivalent to solving a large-scale or ill-condition linear system, it is a natural idea to use the iterative methods such as some Krylov-based methods. Hence, taking the inexactness of the solutions of associated subproblems into account, a more practical inexact heterogeneous ADMM (ihADMM) is proposed. Different from the classical ADMM, we utilize two different weighted inner products to define the augmented Lagrangian function for two subproblems, respectively.
Specifically, based on the $M_h$-weighted inner product, the augmented Lagrangian function with respect to the $u$-subproblem in $k$-th iteration is defined as
\begin{equation*}
  \mathcal{L}_\sigma(u,z^k;\lambda^k)=f(u)+g(z^k)+\langle\lambda,M_h(u-z^k)\rangle+\frac{\sigma}{2}\|u-z^k\|_{M_h}^{2},
\end{equation*}
where $M_h$ is the mass matrix. On the other hand, for the $z$-subproblem, based on the $W_h$-weighted inner product, the augmented Lagrangian function in $k$-th iteration is defined as
\begin{equation*}
  \mathcal{L}_\sigma(u^{k+1},z;\lambda^k)=f(u^{k+1})+g(z)+\langle\lambda,M_h(u^{k+1}-z)\rangle+\frac{\sigma}{2}\|u^{k+1}-z\|_{W_h}^{2},
\end{equation*}
where the lumped mass matrix $W_h$ is diagonal.

As will be mentioned in the Section \ref{sec:4}, benefiting from different weighted techniques, each subproblem of ihADMM for (\ref{equ:approx discretized matrix-vector form}) can be efficiently solved. Specifically, the $u$-subproblem of ihADMM, which result in a large scale linear system, is the main computation cost in whole algorithm. $M_h$-weighted technique could help us to reduce the block three-by-three system to a block two-by-two system without any computational cost so as to reduce calculation amount. On the other hand, $W_h$-weighted technique makes $z$-subproblem have a decoupled form and admit a closed form solution given by the soft thresholding operator and the projection operator onto the box constraint $[a,b]$. Moreover, global convergence and the iteration complexity result $o(1/k)$ in non-ergodic sense for our ihADMM will be proved.
Taking the precision of discretized error into account, we should mention that using our ihADMM algorithm to solve problem (\ref{equ:approx discretized matrix-vector form}) is highly enough and efficient in obtaining an approximate solution with moderate accuracy.

Furthermore, in order to obtain more accurate solutions, if necessarily required, combining ihADMM and semismooth Newton methods together, we give a two-phase strategy. Specifically, our ihADMM algorithm as the Phase-I is used to generate a reasonably good initial point to warm-start Phase-II. In Phase-II, the PDAS method as a postprocessor of our ihADMM is employed to solve the discrete problem to high accuracy.

The remainder of the paper is organized as follows. In Section \ref{sec:2}, an inexact ADMM algorithm in function space for solving (\ref{eqn:orginal problems}) is described. In Section \ref{sec:3}, the finite element approximation is introduced and priori error estimates are proved. In Section \ref{sec:4}, an inexact heterogeneous ADMM (ihADMM) is proposed for the discretized problem. And as the Phase-II algorithm, the PDAS method is also presented. In Section \ref{sec:5},
numerical results are given to confirm the finite element error estimates and show the efficiency of our ihADMM and the two-phase strategy. Finally, we conclude our paper in Section \ref{sec:6}.

\section{An inexact ADMM for (\ref{eqn:modified problems}) in function Space}
\label{sec:2}
In this paper, we assume the elliptic PDEs involved in problem (\ref{eqn:modified problems}).
\begin{equation}\label{eqn:state equations}
\begin{aligned}
   Ay&=u+y_c \qquad \mathrm{in}\  \Omega, \\
   y&=0  \qquad \qquad \mathrm{on}\ \partial\Omega,
\end{aligned}
\end{equation}
satisfy the following assumption.
\begin{assumption}\label{equ:assumption:1}
The linear second-order differential operator $A$ is defined by
 \begin{equation}\label{operator A}
   (Ay)(x):=-\sum \limits^{n}_{i,j=1}\partial_{x_{j}}(a_{ij}(x)y_{x_{i}})+c_0(x)y(x),
 \end{equation}
where functions $a_{ij}(x), c_0(x)\in L^{\infty}(\Omega)$, $c_0\geq0$,
and it is uniformly elliptic, i.e. $a_{ij}(x)=a_{ji}(x)$ and there is a constant $\theta>0$ such that
\begin{equation}\label{equ:operator A coercivity}
  \sum \limits^{n}_{i,j=1}a_{ij}(x)\xi_i\xi_j\geq\theta\|\xi\|^2 \quad \mathrm{for\ almost\ all}\ x\in \Omega\  \mathrm{and\  all}\ \xi \in \mathbb{R}^n.
\end{equation}
\end{assumption}

Then, the weak formulation of (\ref{eqn:state equations}) is given by
\begin{equation}\label{eqn:weak form}
  \mathrm{Find}\ y\in H_0^1(\Omega):\ a(y,v)=(u+y_c,v)_{L^2(\Omega)}\quad \mathrm{for}\ \forall v \in H_0^1(\Omega),
\end{equation}
with the bilinear form
\begin{equation}\label{eqn:bilinear form}
  a(y,v)=\int_{\Omega}(\sum \limits^{n}_{i,j=1}a_{ji}y_{x_{i}}v_{x_{i}}+c_0yv)\mathrm{d}x.
\end{equation}
\begin{proposition}{\rm\textbf{\cite [Theorem B.4] {KiSt}}}\label{equ:weak formulation}
Under Assumption {\ref{equ:assumption:1}}, the bilinear form $a(\cdot,\cdot)$ in {\rm (\ref{eqn:bilinear form})} is bounded and $V$-coercive for $V=H^1_0(\Omega)$. In particular, for every $u \in L^2(\Omega)$ and $y_c\in L^2(\Omega)$, {\rm (\ref{eqn:state equations})} has a unique weak solution $y\in H^1_0(\Omega)$ given by {\rm (\ref{eqn:weak form})}. Furthermore,
\begin{equation}\label{equ:control estimats state}
  \|y\|_{H^1}\leq C (\|u\|_{L^2(\Omega)}+\|y_c\|_{L^2(\Omega)}),
\end{equation}
for a constant $C$ depending only on $a_{ij}$, $c_0$ and $\Omega$.
\end{proposition}

By Proposition \ref{equ:weak formulation}, the solution operator $\mathcal{S}$: $H^{-1}(\Omega)\rightarrow H^1_0(\Omega)$ with $y(u):=\mathcal{S}(u+y_c)$ is well-defined and called the control-to-state mapping, which is a continuous linear injective operator. Since $H_0^1(\Omega)$ is a Hilbert space, the adjoint operator $\mathcal{S^*}$: $H^{-1}(\Omega)\rightarrow H_0^1(\Omega)$ is also a continuous linear operator.

It is clear that problem (\ref{eqn:modified problems}) is continuous and strongly convex . Therefore, the existence and uniqueness of solution of (\ref{eqn:modified problems}) is obvious. The optimal solution can be characterized by the following Karush-Kuhn-Tucker (KKT) conditions:

\begin{theorem}[{\rm First-Order Optimality Condition}]\label{First-Order Optimality Condition}
Under Assumption \ref{equ:assumption:1}, {\rm($y^*$, $u^*$, $z^*$)} is the optimal solution of {\rm(\ref{eqn:modified problems})}, if and only if there exists adjoint state $p^*\in H_0^1(\Omega)$ and Lagrange multiplier $\lambda^*\in L^2(\Omega)$, such that the following conditions hold in the weak sense
\begin{subequations}\label{eqn:KKT}
\begin{eqnarray}
&&\begin{aligned}\label{eqn1:KKT}
        y^*=\mathcal{S}(u^*+y_c),
        \end{aligned}  \\
&& \begin{aligned}\label{eqn2:KKT}
        p^*&=\mathcal{S}^*(y_d-y^*),
        \end{aligned}\\
&&\frac{\alpha}{2} u^*-p^*+\lambda^*=0,\label{eqn3:KKT}\\
&&u^*=z^*,\label{eqn4:KKT}\\
&&{z^*} \in U_{ad},\label{eqn5:KKT}\\
&&{\left\langle\frac{\alpha}{2} z^*-\lambda^*,\tilde{z}-z^*\right\rangle_{L^2(\Omega)}+\beta(\|\tilde{z}\|_{L^1(\Omega)}-\|z^*\|_{L^1(\Omega)})}\geq0,\quad \forall \tilde{z} \in U_{ad}.\label{eqn6:KKT}
\end{eqnarray}
\end{subequations}
Moreover, 
we have
\begin{equation}\label{equ:z-piecepoint form}
  u^*=\mathrm{\Pi}_{U_{ad}}\left(\frac{1}{\alpha}{\rm{soft}}\left(p^*,\beta\right)\right),
\end{equation}
%
where the projection operator $\mathrm{\Pi}_{U_{ad}}(\cdot)$ and the soft thresholding operator $\rm {soft}(\cdot,\cdot)$ are defined as follows, respectively,
\begin{eqnarray}
 \mathrm{\Pi}_{U_{ad}}(v(x))&:=&\max\{a,\min\{v(x),b\}\}, \quad {\rm soft}(v(x),\beta):={\rm{sgn}}(v(x))\circ\max(|v(x)|-\beta,0)\label{softtheresholds}.
\end{eqnarray}
In addition, the optimal control $u$ has the regularity $u\in H^1(\Omega)$.
\end{theorem}

As one may know, ADMM is a simple but powerful algorithm.
Next, we will introduce the ADMM in function space.
Focusing on the ADMM algorithm in function space will help us to better understand the inherent structure. And then it will help us to propose an appropriate discretization scheme and giving a suitable ADMM-type algorithm to solve the corresponding discretized problem.

Using the operator $\mathcal{S}$, the problem (\ref{eqn:modified problems}) can be equivalently rewritten as the following form:
\begin{equation}\label{eqn:reduced form}
  \left\{ \begin{aligned}
        &\min \limits_{u, z}^{}\ \ f(u)+g(z)\\
        &~~{\rm{s.t.}}\quad  u=z,
                          \end{aligned} \right.\tag{$\mathrm{RP}$}
\end{equation}
with the reduced cost function
\begin{eqnarray}
  f(u):&=&\frac{1}{2}\|\mathcal{S}(u+y_c)-y_d\|_{L^2(\Omega)}^{2}+\frac{\alpha}{4}\|u\|_{L^2(\Omega)}^{2},\\
  g(z):&=&\frac{\alpha}{4}\|z\|_{L^2(\Omega)}^{2}+\beta\|z\|_{L^1(\Omega)}+\delta_{U_{ad}}(z).
\end{eqnarray}
Let us define the augmented Lagrangian function of (\ref{eqn:reduced form}) as follows:
\begin{equation}\label{augmented Lagrangian function}
  \mathcal{L}_\sigma(u,z;\lambda)=f(u)+g(z)+\langle\lambda,u-z\rangle_{L^2(\Omega)}+\frac{\sigma}{2}\|u-z\|_{L^2(\Omega)}^{2}
\end{equation}
with the Lagrange multiplier $\lambda \in L^2(\Omega)$ and $\sigma>0$ be a penalty parameter. Moreover, for the convergence property and the iteration complexity analysis, we define the function $R: (u,z,\lambda)\rightarrow [0,\infty)$ by:
\begin{equation}\label{KKT function}
  R(u,z,\lambda):=\|\nabla f(u)+\lambda\|^2_{L^2(\Omega)}+{\rm dist}^2(0, -\lambda+\partial g(z))+\|u-z\|^2_{L^2(\Omega)}
\end{equation}
Then, the iterative scheme of inexact ADMM for the problem (\ref{eqn:reduced form}) is shown in Algorithm \ref{algo1:ADMM for problems RP}.

\begin{algorithm}[H]
\caption{inexact ADMM algorithm for (\ref{eqn:reduced form})}\label{algo1:ADMM for problems RP}
  \textbf{Input}: {$(z^0, u^0, \lambda^0)\in {\rm dom} (\delta_{U_{ad}}(\cdot))\times L^2(\Omega) \times L^2(\Omega)$ and a parameter $\tau \in (0,\frac{1+\sqrt{5}}{2})$. Let $\{\epsilon_k\}^\infty_{k=0}$ be a sequence satisfying $\{\epsilon_k\}^\infty_{k=0}\subseteq [0,+\infty)$ and $\sum\limits_{k=0}^{\infty}\epsilon_k<\infty$. Set $k=0$.}\\
  \textbf{Output}: {$ u^k, z^{k}, \lambda^k$}
\begin{description}
\item[Step 1] Find an minizer (inexact)
\begin{equation*}\label{u-subproblem}
\begin{aligned}
  u^{k+1}&=\arg\min \mathcal{L}_\sigma(u,z^k;\lambda^k)-\langle\delta^k, u\rangle_{L^2(\Omega)},
\end{aligned}
\end{equation*}
where the error vector ${\delta}^k$ satisfies $\|{\delta}^k\|_{L^2(\Omega)} \leq {\epsilon_k}$.
\item[Step 2] Compute $z^k$ as follows:
\begin{equation*}\label{z-subproblem}
\begin{aligned}
       z^k&=\arg\min\mathcal{L}_\sigma(u^{k+1},z;\lambda^k).
\end{aligned}
\end{equation*}
  \item[Step 3] Compute
  \begin{eqnarray*}
    \lambda^{k+1} &=& \lambda^k+\tau\sigma(u^{k+1}-z^{k+1}).
  \end{eqnarray*}
  \item[Step 4] If a termination criterion is not met, set $k:=k+1$ and go to Step 1.
\end{description}
\end{algorithm}
About the global convergence as well as the iteration complexity of the inexact ADMM for (\ref{eqn:modified problems}), we have the following results.
\begin{theorem}
Suppose that Assumption {\rm\ref{equ:assumption:1}} holds. Let $(y^*,u^*,z^*,p^*,\lambda^*)$ is the KKT point of {\rm(\ref{eqn:modified problems})} which satisfies {\rm(\ref{eqn:KKT})}, the sequence $\{(u^{k},z^{k},\lambda^k)\}$ is generated by Algorithm \ref{algo1:ADMM for problems RP} with the associated state $\{y^k\}$ and adjoint state $\{p^k\}$, then we have
\begin{eqnarray*}
  &&\lim\limits_{k\rightarrow\infty}^{}\{\|u^{k}-u^*\|_{L^{2}(\Omega)}+\|z^{k}-z^*\|_{L^{2}(\Omega)}+\|\lambda^{k}-\lambda^*\|_{L^{2}(\Omega)} \}= 0,\\
   && \lim\limits_{k\rightarrow\infty}^{}\{\|y^{k}-y^*\|_{H_0^{1}(\Omega)}+\|p^{k}-p^*\|_{H_0^{1}(\Omega)} \}= 0.
\end{eqnarray*}
Moreover, there exists a constant $C$ only depending on the initial point ${(u^0,z^0,\lambda^0)}$ and the optimal solution ${(u^*,z^*,\lambda^*)}$ such that for $k\geq1$,
\begin{eqnarray}
  &&\min\limits^{}_{1\leq i\leq k} \{R(u^i,z^i,\lambda^i)\}\leq\frac{C}{k}, \quad \lim\limits^{}_{k\rightarrow\infty}\left(k\times\min\limits^{}_{1\leq i\leq k} \{R(u^i,z^i,\lambda^i)\}\right) =0.
\end{eqnarray}
where $R(\cdot)$ is defined as in {\rm(\ref{KKT function})}
\end{theorem}
\begin{proof}
The proof is a direct application of the general inexact ADMM in Hilbert Space for the problem (\ref{eqn:reduced form}) and omitted here. We refer the reader to literatures \cite{Boyd,inexactADMM}.
\end{proof}

\begin{remark}
1). The first subproblems of Algorithm {\rm\ref{algo1:ADMM for problems RP}} is a convex differentiable optimization problem with respect to $u$, if we omit the error vector $\delta^k$, thus it is equivalent to solving the following system{\rm:}
\begin{equation}\label{equ:saddle point problems1}
\left[
  \begin{array}{ccc}
    I & 0 & \quad\mathcal{S}^{-*} \\
    0 & (\frac{\alpha}{2}+\sigma)I & \quad-I \\
    \mathcal{S}^{-1} & -I & \quad0 \\
  \end{array}
\right]\left[
         \begin{array}{c}
           y^{k+1} \\
           u^{k+1} \\
           p^{k+1} \\
         \end{array}
       \right]=\left[
                 \begin{array}{c}
                   y_d \\
                   \sigma z^k-\lambda^k \\
                   y_c \\
                 \end{array}
               \right],
\end{equation}
Moreover, 
we could eliminate the variable $p$ and derive the following reduced system{\rm:}
\begin{equation}\label{equ:saddle point problems2}
\left[
  \begin{array}{cc}
    (\frac{\alpha}{2}+\sigma)I & \quad \mathcal{S}^* \\
    -\mathcal{S} & \quad I\\
  \end{array}
\right]\left[
         \begin{array}{c}
           u^{k+1} \\
           y^{k+1} \\
         \end{array}
       \right]=\left[
                 \begin{array}{c}
                   \mathcal{S}^*y_d+\sigma z^k-\lambda^k \\
                   \mathcal{S}y_c,\\
                 \end{array}
               \right]
\end{equation}
where $I$ represents the identity operator.

2). It is easy to see that $z$-subproblem has a closed solution:
\begin{equation}\label{z-closed form solution}
\begin{aligned}
z^{k+1}
       &=\mathrm{\Pi}_{U_{ad}}\left(\frac{1}{\gamma}{\rm{soft}}\left(\sigma u^{k+1}+\lambda^k,\beta\right)\right),
\end{aligned}
\end{equation}
where $\gamma=0.5\alpha+\sigma$.
\end{remark}

Based on the well-structure of (\ref{equ:saddle point problems2}) and (\ref{z-closed form solution}), it will be a crucial point in the numerical analysis to establish relations parallel to
(\ref{equ:saddle point problems2}) and (\ref{z-closed form solution}) also for the discretized problem.

\section{Finite Element Approximation} 
\label{sec:3}
The goal of this section is to study the approximation of problems (\ref{eqn:orginal problems}) and (\ref{eqn:modified problems}) by finite elements. 

To achieve our aim, we first consider a family of regular and quasi-uniform triangulations $\{\mathcal{T}_h\}_{h>0}$ of $\bar{\Omega}$. For each cell $T\in \mathcal{T}_h$, let us define the diameter of the set $T$ by $\rho_{T}:={\rm diam}\ T$ and define $\sigma_{T}$ to be the diameter of the largest ball contained in $T$. The mesh size of the grid is defined by $h=\max_{T\in \mathcal{T}_h}\rho_{T}$. We suppose that the following regularity assumptions on the triangulation are satisfied which are standard in the context of error estimates.
\begin{assumption}\label{assumption on mesh}
There exist two positive constants $\kappa$ and $\tau$ such that
   \begin{equation*}
   \frac{\rho_{T}}{\sigma_{T}}\leq \kappa \quad and\quad \frac{h}{\rho_{T}}\leq \tau,
 \end{equation*}
hold for all $T\in \mathcal{T}_h$ and all $h>0$.
Let us define $\bar{\Omega}_h=\bigcup_{T\in \mathcal{T}_h}T$, and let ${\Omega}_h \subset\Omega$ and $\Gamma_h$ denote its interior and its boundary, respectively. In the case that $\Omega$ is a convex polyhedral domain, we have $\Omega=\Omega_h$. In the case that $\Omega$ has a $C^{1,1}$- boundary $\Gamma$, we assumed that $\bar{\Omega}_h$ is convex and that all boundary vertices of $\bar{\Omega}_h$ are contained in $\Gamma$, such that
$|\Omega\backslash {\Omega}_h|\leq c h^2$,
where $|\cdot|$ denotes the measure of the set and $c>0$ is a constant.
\end{assumption}

On account of the homogeneous boundary condition of the state equation, we use
\begin{equation*}
  Y_h =\left\{y_h\in C(\bar{\Omega})~\big{|}~y_{h|T}\in \mathcal{P}_1~ {\rm{for\ all}}~ T\in \mathcal{T}_h~ \mathrm{and}~ y_h=0~ \mathrm{in } ~\bar{\Omega}\backslash {\Omega}_h\right\}
\end{equation*}
as the discretized state space, where $\mathcal{P}_1$ denotes the space of polynomials of degree less than or equal to $1$. For a given source term $y_c$ and right-hand side $u\in L^2(\Omega)$, we denote by $y_h(u)$ the approximated state associated with $u$, which is the unique solution for the following discretized weak formulation: 
\begin{equation}\label{eqn:discrete weak solution}
 \int_{\Omega_h}\left(\sum \limits^{n}_{i,j=1}a_{ij}{y_h}_{x_{i}}{v_h}_{x_{j}}+c_0y_hv_h\right)\mathrm{d}x=\int_{\Omega_h}(u+y_c)v_h{\rm{d}}x \qquad  \forall v_h\in Y_h.
\end{equation}
Moreover, $y_h(u)$ can also be expressed by $y_h(u)={\mathcal{S}}_{h}(u+y_c)$, in which  ${\mathcal{S}}_{h}$ is a discretized vision of $\mathcal{S}$ and an injective, selfadjoint operator. The following error estimates are well-known.
\begin{lemma}{\rm\textbf{\cite[Theorem 4.4.6]{Ciarlet}}}\label{eqn:lemma1}
For a given $u\in L^2(\Omega)$, let $y$ and $y_h(u)$ be the unique solution of {\rm(\ref{eqn:weak form})} and {\rm(\ref{eqn:discrete weak solution})}, respectively. Then there exists a constant $c_1>0$ independent of $h$, $u$ and $y_c$ such that

\begin{equation}\label{estimates1}
  \|y-y_h(u)\|_{L^2(\Omega)}+h\|\nabla y-\nabla y_h(u)\|_{L^2(\Omega)}\leq c_1h^2(\|u\|_{L^2(\Omega)}+\|y_c\|_{L^2(\Omega)}).
\end{equation}
In particular, this implies $\|\mathcal{S}-\mathcal{S}_h\|_{L^2\rightarrow L^2}\leq c_1h^2$ and $\|\mathcal{S}-\mathcal{S}_h\|_{L^2\rightarrow H^1}\leq c_1h$.
\end{lemma}

Considering the homogeneous boundary condition of the adjoint state equation (\ref{eqn:state equations}) and the projection formula (\ref{equ:z-piecepoint form}), we use
\begin{equation*}
   U_h =\left\{u_h\in C(\bar{\Omega})~\big{|}~u_{h|T}\in \mathcal{P}_1~ {\rm{for\ all}}~ T\in \mathcal{T}_h~ \mathrm{and}~ u_h=0~ \mathrm{in } ~\bar{\Omega}\backslash{\Omega}_h\right\},
\end{equation*}
as the discretized space of the control $u$ and artificial variable $z$.

For a given regular and quasi-uniform triangulation $\mathcal{T}_h$ with nodes $\{x_i\}_{i=1}^{N_h}$, let $\{\phi_i(x)\} _{i=1}^{N_h}$ be a set of nodal basis functions associated with nodes $\{x_i\}_{i=1}^{m}$, where the basis functions satisfy the following properties:
\begin{equation}\label{basic functions properties}
   \phi_i(x) \geq 0, \quad \|\phi_i(x)\|_{\infty} = 1 \quad \forall i=1,2,...,N_h,\quad \sum\limits_{i=1}^{N_h}\phi_i(x)=1.
\end{equation}
The elements $z_h \in U_h$, $u_h\in U_h$ and $y_h\in Y_h$, can be represented in the following forms, respectively,
\begin{equation*}
  u_h=\sum \limits_{i=1}^{N_h}u_i\phi_i(x), \quad z_h=\sum \limits_{i=1}^{N_h}z_i\phi_i(x),\quad y_h=\sum \limits_{i=1}^{N_h}y_i\phi_i(x),
\end{equation*}
and $u_h(x_i)=u_i$, $z_h(x_i)=z_i$ and $y_h(x_i)=y_i$ hold.

Let $U_{ad,h}$ denotes the discretized feasible set, which is defined by
\begin{eqnarray*}
  U_{ad,h}:&=&U_h\cap U_{ad}
           =\left\{z_h=\sum \limits_{i=1}^{N_h}z_i\phi_i(x)~\big{|}~a\leq z_i\leq b, \forall i=1,...,N_h\right\}\subset U_{ad}.
\end{eqnarray*}
Following the approach of \cite{Cars}, for the error analysis further below, let us introduce a quasi-interpolation operator $\Pi_h:L^1(\Omega_h)\rightarrow U_h$ which provides interpolation estimates. For an arbitrary $w\in L^1(\Omega)$, the operator $\Pi_h$ is constructed as follows:
\begin{equation}\label{equ:quasi-interpolation}
   \Pi_hw=\sum\limits_{i=1}^{N_h}\pi_i(w)\phi_i(x), \quad \pi_i(w)=\frac{\int_{\Omega_h}w(x)\phi_i(x){\rm{d}}x}{\int_{\Omega_h}\phi_i(x){\rm{d}}x}.
\end{equation}
And we know that:
\begin{equation}\label{quasi-interpolation property}
 w\in U_{ad} \Rightarrow \Pi_hw \in U_{ad,h}, \quad {\rm for\  all}\ w\in L^1(\Omega).
\end{equation}
Based on the assumption on the mesh and the control discretization , we extend $\Pi_hw$ to $\Omega$ by taking $\Pi_hw=w$ for every $x\in\Omega\backslash {\Omega}_h$, and have the following estimates of the interpolation error. For the detailed proofs, we refer to \cite{Cars,MeReVe}.
\begin{lemma}\label{eqn:lemma3}
There is a constant $c_2$ independent of $h$ such that
\begin{equation*}
  h\|z-\Pi_hz\|_{L^2(\Omega)}+\|z-\Pi_hz\|_{H^{-1}(\Omega)}\leq c_2h^2\|z\|_{H^1(\Omega)},
\end{equation*}
holds for all $z\in H^1(\Omega)$.
\end{lemma}

Now, we can consider a discretized version of problem (\ref{eqn:modified problems}) as:
\begin{equation}\label{equ:separable discrete problem}
  \left\{ \begin{aligned}
        &\min J_h(y_h,u_h,z_h)=\frac{1}{2}\|y_h-y_d\|_{L^2(\Omega_h)}^{2}+\frac{\alpha}{4}\|u_h\|_{L^2(\Omega_h)}^{2}
       +\frac{\alpha}{4}\|z_h\|_{L^{2}(\Omega_h)}^{2}+\beta\|z_h\|_{L^{1}(\Omega_h)} \\
        &\qquad\qquad {\rm{s.t.}}\qquad \quad~~ y_h=\mathcal{S}_h(u_h+y_c),  \\
         &\qquad \qquad \qquad \quad\qquad   u_h=z_h,\\
          &\qquad \qquad \qquad \quad\qquad z_h\in U_{ad,h},
                          \end{aligned} \right.\tag{$\mathrm{\widetilde{P}}_{h}$}
 \end{equation}
where
\begin{eqnarray}
  \label{eqn:exact norm1}\|z_h\|^2_{L^2(\Omega_h)}&=& \int_{\Omega_h}\left(\sum\limits_{i=1}^{N_h}z_i\phi_i(x)\right)^2\mathrm{d}x, \\
  \label{eqn:exact norm2}\|z_h\|_{L^1(\Omega_h)}&=&\int_{\Omega_h}\big{|}\sum\limits_{i=1}^{N_h}z_i\phi_i(x)\big{|}\mathrm{d}x.
\end{eqnarray}
This implies, for problem (\ref{eqn:orginal problems}), we have the following discretized version:
\begin{equation}\label{equ:discrete problem}
  \left\{ \begin{aligned}
        &\min \limits_{(y_h,u_h,z_h)\in Y_h\times U_h\times U_h}^{}J_h(y_h,u_h,z_h)=\frac{1}{2}\|y_h-y_d\|_{L^2(\Omega_h)}^{2}+\frac{\alpha}{2}\|u_h\|_{L^2(\Omega_h)}^{2}+\beta\|u_h\|_{L^{1}(\Omega_h)} \\
        &\qquad\qquad {\rm{s.t.}}\qquad \quad~~ y_h=\mathcal{S}_h(u_h+y_c),  \\
          &\qquad \qquad \qquad \quad\qquad u_h\in U_{ad,h}.
                          \end{aligned} \right.\tag{$\mathrm{P}_{h}$}
 \end{equation}
For problem (\ref{equ:discrete problem}), in \cite{WaWa}, the authors gave the following error estimates results.
\begin{theorem}{\rm\textbf{\cite[Proposition 4.3]{WaWa}}}\label{theorem:error1}
Let $(y, u)$ be the optimal solution of problem {\rm(\ref{eqn:orginal problems})}, and $(y_h, u_h)$ be the optimal solution of problem {\rm(\ref{equ:discrete problem})}. For every $h_0>0$, $\alpha_0>0$, there is a constant $C>0$ such that for all $0<\alpha\leq\alpha_0$, $0<h\leq h_0$ it holds
\begin{equation}\label{u-error-estimates}
  \|u-u_h\|_{L^2(\Omega)}\leq C(\alpha^{-1}h+\alpha^{-\frac{3}{2}}h^2),
\end{equation}
where $C$ is a constant independent of $h$ and $\alpha$.
\end{theorem}

However, the resulting discretized problem (\ref{equ:separable discrete problem}) is not in a decoupled form as the finite dimensional $l^1$-regularization optimization problem usually does, since (\ref{eqn:exact norm1}) and (\ref{eqn:exact norm2})
do not have a decoupled form. Thus, if we directly apply ADMM algorithm to solve the discretized problem, then the $z$-subproblem can not have a closed form solution which similar to (\ref{z-closed form solution}). 
Thus, directly solving (\ref{equ:separable discrete problem}) it can not make full use of the advantages of ADMM.
In order to overcome this bottleneck, we introduce the nodal quadrature formulas to approximately discretized the \textsl{$L^2$}-norm and \textsl{$L^1$}-norm. Let
\begin{eqnarray}
&&\|z_h\|_{L^{2}_h(\Omega_h)}:=\left(\sum\limits_{i=1}^{N_h}(z_i)^2\int_{\Omega_h}\phi_i(x)\mathrm{dx}\right)^\frac{1}{2},\label{eqn:approx norm2}\\
&&\|z_h\|_{L^{1}_h(\Omega_h)}:=\sum\limits_{i=1}^{N_h}|z_i|\int_{\Omega_h}\phi_i(x)\mathrm{dx},\label{eqn:approx norm1}
\end{eqnarray}
and call them $L^{2}_h$- and $L^{1}_h$-norm, respectively.

It is obvious that the \textsl{$L^{2}_h$}-norm and the \textsl{$L^{1}_h$}-norm can be considered as a weighted $l^2$-norm and a weighted $l^1$-norm of the coefficient of $z_h$, respectively. Both of them are norms on $U_h$. In addition, the \textsl{$L^{2}_h$}-norm is a norm induced by the following inner product:
\begin{equation}\label{eqn:approx inner product}
 \langle z_h,v_h\rangle_{L^{2}_h(\Omega_h)}=\sum\limits_{i=1}^{N_h}(z_iv_i)\int_{\Omega_h}\phi_i(x)\mathrm{d}x\quad {\rm{for}}\  z_h,v_h\in U_h.
\end{equation}\\
More importantly, the following properties hold.
\begin{proposition}{\rm\textbf{\cite[Table 1]{Wathen}}}\label{eqn:martix properties}
$\forall$ $z_h\in U_h$, the following inequalities hold:
\begin{eqnarray}
 \label{equ:martix properties1}&&\|z_h\|^2_{L^{2}(\Omega_h)}\leq\|z_h\|^2_{L^{2}_h(\Omega_h)}\leq c\|z_h\|^2_{L^{2}(\Omega_h)}, \quad where \quad c=
 \left\{ \begin{aligned}
         &4  \quad if \quad n=2, \\
         &5  \quad if \quad n=3.
                           \end{aligned} \right.
                           \\
 \label{equ:martix properties2} &&\int_{\Omega_h}|\sum_{i=1}^n{z_i\phi_i(x)}|~\mathrm{d}x\leq\|z_h\|_{L^{1}_h(\Omega_h)}.
\end{eqnarray}
\end{proposition}
Thus, based on (\ref{eqn:approx norm1}) and (\ref{eqn:approx norm2}), we derive a new discretized optimal control problems 
\begin{equation}\label{equ:approx discretized}
  \left\{ \begin{aligned}
        &\min J_h(y_h,u_h,z_h)=\frac{1}{2}\|y_h-y_d\|_{L^2(\Omega_h)}^{2}+\frac{\alpha}{4}\|u_h\|_{L^2(\Omega_h)}^{2}
        +\frac{\alpha}{4}\|z_h\|_{L^{2}_h(\Omega_h)}^{2}+\beta\|z_h\|_{L^{1}_h(\Omega_h)} \\
        &\qquad\quad \quad{\rm{s.t.}}\qquad \quad~~ y_h=\mathcal{S}_hu_h,  \\
         &\qquad \qquad \qquad\qquad \quad   u_h=z_h,\\
          &\qquad \qquad \qquad\qquad \quad z_h\in U_{ad,h}.
                          \end{aligned} \right.\tag{$\mathrm{\widetilde{DP}}_{h}$}
 \end{equation}

It is should mentioned that the approximate $L^{1}_h$ was already used in \cite[Section 4.4]{WaWa}. However, different from their discretization schemes, in this paper, in order to keep the separability of the discrete $L^2$-norm with respect to $z$, we use (\ref{eqn:approx norm2}) to approximately discretize it.
In addition, although these nodal quadrature formulas incur additional discrete errors, as it will be proven that these approximation steps will not change the order of error estimates as shown in (\ref{u-error-estimates}), see Theorem \ref{theorem:error1}. More importantly, these nodal quadrature formulas will turn out to be crucial in order to obtain formulas parallel to (\ref{equ:saddle point problems2}) and (\ref{z-closed form solution}) for the discretized problem (\ref{equ:approx discretized}), see Remark 4.4 below.

Analogous to the continuous problem (\ref{eqn:modified problems}), the discretized problem (\ref{equ:approx discretized}) is also a strictly convex problem, which is uniquely solvable. We derive the following first-order optimality conditions, which is necessary and sufficient for the optimal solution of (\ref{equ:approx discretized}).
\begin{theorem}[{\rm \textbf{Discrete first-order optimality condition}}]
$(u_h, z_h, y_h)$ is the optimal solution of {\rm{(\ref{equ:approx discretized})}}, if and only if there exist an adjoint state $p_h$ and a Lagrange multiplier $\lambda_h$, such that the following conditions are satisfied
\begin{subequations}\label{eqn:DKKT}
\begin{eqnarray}
&&y_h=\mathcal{S}_h (u_h+y_c)\label{eqn1:DKKT}, \\
&&p_h=\mathcal{S}_h^*(y_h-y_d)\label{eqn2:DKKT},\\
&&\frac{\alpha}{2} u_h+p_h+\lambda_h=0\label{eqn3:DKKT},\\
&&u_h=z_h\label{eqn4:DKKT},\\
&&{z_h} \in U_{ad,h}\label{eqn5:DKKT},\\
&&{\left\langle\frac{\alpha}{2} z_h,\tilde{z}_h-z_h\right\rangle_{L^{2}_h(\Omega_h)}-(\lambda_h,\tilde{z}_h-z_h)_{L^2(\Omega_h)}+\beta\left(\|\tilde{z}_h\|_{L^{1}_h(\Omega_h)}-\|z\|_{L^{1}_h(\Omega_h)}\right)}\geq0, \label{eqn6:DKKT}\\
\nonumber\qquad&&\forall \tilde{z}_h \in U_{ad,h}.
\end{eqnarray}
\end{subequations}
\end{theorem}

Now, let us start to do error estimation. Let $(y, u, z)$ be the optimal solution of problem (\ref{eqn:modified problems}), and $(y_h, u_h, z_h)$ be the optimal solution of problem (\ref{equ:approx discretized}). We have the following results.
\begin{theorem}\label{theorem:error2}
Let $(y, u, z)$ be the optimal solution of problem {\rm(\ref{eqn:modified problems})}, and $(y_h, u_h,z_h)$ be the optimal solution of problem {\rm(\ref{equ:approx discretized})}. For any $h>0$ small enough and $\alpha_0>0$, there is a constant $C$ such that: for all $0<\alpha\leq\alpha_0$,
\begin{eqnarray*}
  \frac{\alpha}{2}\|u-u_h\|^2_{L^2(\Omega)}+\frac{1}{2}\|y-y_h\|^2_{L^2(\Omega)}\leq C(h^2+\alpha h^2+\alpha^{-1} h^2+h^3+\alpha^{-1}h^4+\alpha^{-2}h^4),
\end{eqnarray*} where $C$ is a constant independent of $h$ and $\alpha$.
\end{theorem}

\proof
Due to the optimality of $z$ and $z_h$, $z$ and $z_h$ satisfy (\ref{eqn6:KKT}) and (\ref{eqn6:DKKT}), respectively. Let us use the test function $z_h\in U_{ad,h}\subset U_{ad}$ in (\ref{eqn6:KKT}) and the test function $\tilde{z}_h:=\Pi_hz\in U_{ad,h}$ in (\ref{eqn6:DKKT}), thus we have
\begin{eqnarray}
\label{eqn:error1}&&{\left\langle\frac{\alpha}{2} z-\lambda,z_h-z\right\rangle_{L^2(\Omega)}+\beta\left(\|z_h\|_{L^1(\Omega)}-\|z\|_{L^1(\Omega)}\right)}\geq0, \\
\label{eqn:error2}&&{\left\langle\frac{\alpha}{2} z_h,\tilde{z}_h-z_h\right\rangle_{L^{2}_h(\Omega_h)}-\langle\lambda_h,\tilde{z}_h-z_h\rangle_{L^2(\Omega_h)}+\beta\left(\|\tilde{z}_h\|_{L^{1}_h(\Omega_h)}-\|z_h\|_{L^{1}_h(\Omega_h)}\right)}\geq0.
\end{eqnarray}
Because $z_h=0$ on $\bar\Omega\backslash {\Omega}_h$, the integrals over $\Omega$ can be replaced by integrals over $\Omega_h$ in (\ref{eqn:error1}), and it can be rewritten as
\begin{eqnarray}\label{eqn:error3}
 \hspace{0.3in}{\left\langle\frac{\alpha}{2} z-\lambda,z-z_h\right\rangle_{L^2(\Omega_h)}+\beta\left(\|z\|_{L^1(\Omega_h)}-\|z_h\|_{L^1(\Omega_h)}\right)}&\leq&\left\langle\lambda-\frac{\alpha}{2} z,z\right\rangle_{L^2(\Omega\backslash {\Omega}_h)}-\beta\|z\|_{L^1(\Omega\backslash {\Omega}_h)} \\
  \nonumber&\leq&  \langle\lambda,z\rangle_{L^2(\Omega\backslash {\Omega}_h)}\leq ch^2,
\end{eqnarray}
where the last inequality follows from the boundedness of $\lambda$ and $z$ and the assumption $|\Omega\backslash {\Omega}_h|\leq c h^2$.

By the definition of the quasi-interpolation operator in (\ref{equ:quasi-interpolation}) and (\ref{equ:martix properties1}) in Proposition \ref{eqn:martix properties}, we have
\begin{equation}\label{estimates inner product}
\begin{aligned}
 \langle z_h,\tilde{z}_h- z_h\rangle_{L^{2}_h(\Omega_h)}&=\langle z_h,\tilde{z}_h\rangle_{L^{2}_h(\Omega_h)}-\|z_h\|^2_{L^{2}_h(\Omega_h)}
\leq\langle z_h,z-z_h\rangle_{L^2(\Omega_h)}.
\end{aligned}
\end{equation}
Thus, (\ref{eqn:error2}) can be rewritten as
\begin{eqnarray}
\label{eqn:error4}&&{\left\langle-\frac{\alpha}{2} z_h+\lambda_h,z-z_h\right\rangle_{L^{2}(\Omega_h)}+\langle\lambda_h,\tilde{z}_h-z\rangle_{L^2(\Omega_h)}-\beta\left(\|\tilde{z}_h\|_{L^{1}_h(\Omega_h)}-\|z_h\|_{L^{1}_h(\Omega_h)}\right)}\leq0.
\end{eqnarray}
Adding up and rearranging (\ref{eqn:error3}) and (\ref{eqn:error4}), we obtain
\begin{equation}\label{eqn:error estimat1}
  \begin{aligned}
  \frac{\alpha}{2}\|z-z_h\|^2_{L^2(\Omega_h)}\leq& \langle\lambda-\lambda_h,z-z_h\rangle_{L^2(\Omega_h)}-\langle\lambda_h,\tilde{z}_h-z\rangle_{L^2(\Omega_h)}\\
&+\beta\left(\|z_h\|_{L^1(\Omega_h)}-\|z\|_{L^1(\Omega_h)}+\|\tilde{z}_h\|_{L^{1}_h(\Omega_h)}-\|z_h\|_{L^{1}_h(\Omega_h)}\right)+ch^2\\
\leq&\begin{array}{c}
\underbrace{\left\langle\frac{\alpha}{2}(u_h-u)+p_h-p,z-z_h\right\rangle_{L^2(\Omega_h)}}\\
I_1                                                                                 \end{array}
-\begin{array}{c}
\underbrace{\left\langle\frac{\alpha}{2}u_h+p_h,\tilde{z}_h-z\right\rangle_{L^2(\Omega_h)}}\\
I_2
\end{array}\\
&+\begin{array}{c}
\underbrace{\beta\left(\|z_h\|_{L^1(\Omega_h)}-\|z\|_{L^1(\Omega_h)}+\|\tilde{z}_h\|_{L^{1}_h(\Omega_h)}-\|z_h\|_{L^{1}_h(\Omega_h)}\right)}\\
I_3
\end{array}
+ch^2,
\end{aligned}
\end{equation}
where the second inequality follows from (\ref{eqn3:KKT}) and (\ref{eqn3:DKKT}).

Next, we first estimate the third term $I_3$. By (\ref{equ:martix properties2}) in Proposition \ref{eqn:martix properties}, we have $\|z_h\|_{L^1(\Omega_h)}\leq\|z_h\|_{L^{1}_h(\Omega_h)}$.
And following from the definition of $\tilde{z}_h=\Pi_h(z)$ and the non-negativity and partition of unity of the nodal basis functions, we get
\begin{equation}\label{equ:equation about l^1 norm}
  \|\tilde{z}_h\|_{L^{1}_h(\Omega_h)}=\|\Pi_h(z)\|_{L^{1}_h(\Omega_h)}
=\sum\limits_{i=1}^{N_h}\left|\frac{\int_{\Omega_h}z(x)\phi_i{\rm{d}}x}{\int_{\Omega_h}\phi_i{\rm{d}}x}\right|{\int_{\Omega_h}\phi_i{\rm{d}}x}=\|z\|_{L^1(\Omega_h)}.
\end{equation}
Thus, we have $I_3\leq 0$.

For the terms $I_1$ and $I_2$, from $u=z$,~$u_h=z_h$, we get
\begin{equation*}
  I_1-I_2=-\frac{\alpha}{2}\|u-u_h\|^2_{L^2(\Omega_h)}+\langle p_h-p,\tilde{z}_h-z_h\rangle_{L^2(\Omega_h)}+\left\langle\frac{\alpha}{2}u+p,\tilde{z}_h-z\right\rangle_{L^2(\Omega_h)}+
  \frac{\alpha}{2}\langle u_h-u,\tilde{z}_h-z\rangle_{L^2(\Omega_h)}.
\end{equation*}
Then (\ref{eqn:error estimat1}) can be rewritten as
\begin{equation}\label{error estimat12}
  \begin{aligned}
 \frac{\alpha}{2}\|z-z_h\|^2_{L^2(\Omega_h)}+\frac{\alpha}{2}\|u-u_h\|^2_{L^2(\Omega_h)}&\leq
 \begin{array}{c}
 \underbrace{\langle p_h-p,\tilde{z}_h-z_h\rangle_{L^2(\Omega_h)}}\\
 I_4
 \end{array}
+\begin{array}{c}
 \underbrace{\left\langle\frac{\alpha}{2}u+p,\tilde{z}_h-z\right\rangle_{L^2(\Omega_h)}}\\
I_5
\end{array}\\
  &+
  \begin{array}{c}
 \underbrace{\frac{\alpha}{2}\langle u_h-u,\tilde{z}_h-z\rangle_{L^2(\Omega_h)}}\\
 I_6
 \end{array}+ch^2.
   \end{aligned}
\end{equation}
For the term $I_4$, let $\tilde{p}_h=\mathcal{S}^*_h(y-y_d)$, we have
\begin{align*}
  I_4&=\langle p_h-\tilde{p}_h+\tilde{p}_h-p,\tilde{z}_h-z_h\rangle_{L^2(\Omega_h)} \\
     &=-\|y-y_h\|^2_{L^2(\Omega_h)}+\begin{array}{c}
 \underbrace{\langle y_h-y,(\mathcal{S}_h-\mathcal{S})(\tilde{z}_h+y_c)-\mathcal{S}(z-\tilde{z}_h)\rangle_{L^2(\Omega_h)}}\\
 I_7
\end{array}\\
 &\quad +      \begin{array}{c}
 \underbrace{(y-y_d,(\mathcal{S}_h-\mathcal{S})(\tilde{z}_h-z_h))_{L^2(\Omega_h)}}\\
 I_8
\end{array}
.
\end{align*}
Consequently,
\begin{equation}\label{error estimat13}
  \frac{\alpha}{2}\|z-z_h\|^2_{L^2(\Omega_h)}+\frac{\alpha}{2}\|u-u_h\|^2_{L^2(\Omega_h)}+\|y-y_h\|^2_{L^2(\Omega_h)}\leq I_5+I_6+I_7+I_8+ch^2.
\end{equation}

In order to further estimate (\ref{error estimat13}), we will discuss each of these items from $I_5$ to $I_8$ in turn. Firstly, from the regularity of the optimal control $u$, i.e., $u\in H^1(\Omega)$, and (\ref{equ:z-piecepoint form}), we know that
\begin{equation}\label{eqn:exact function estimats1}
  \|u\|_{H^1(\Omega)}\leq \frac{1}{\alpha}\|p\|_{H^1(\Omega)}+\left(\frac{\beta}{\alpha}+a+b\right)\mathcal{M}(\Omega),
\end{equation}
where $\mathcal{M}(\Omega)$ denotes the measure of the $\Omega$. Then we have
\begin{equation*}
  \|\frac{\alpha}{2}u+p\|_{H^1(\Omega)}\leq\frac{3}{2}\|p\|_{H^1(\Omega)}+\frac{1}{2}(\beta+\alpha a+\alpha b)\mathcal{M}(\Omega).
\end{equation*}
Moreover, due to the boundedness of the optimal control $u$, the state $y$, the adjoint state $p$ and the operator $\mathcal{S}$, we can choose a large enough constant $L>0$ independent of $\alpha$, $h$ and a constant $\alpha_0$, such that for all $0<\alpha\leq\alpha_0$ and $h>0$, the following inequation holds:
\begin{equation}\label{eqn:exact function estimats2}
  \frac{3}{2}\|p\|_{H^1(\Omega)}+(\beta+\alpha a+\alpha b)\mathcal{M}(\Omega)+\|y-y_d\|_{L^2(\Omega)}+\|y_c\|_{L^2(\Omega)}+\|\mathcal{S}\|_{\mathcal{L}(H^{-1},L^2)}+\sup\limits_{u_h\in U_{ad,h}}{}\|u_h\|\leq L.
\end{equation}
From (\ref{eqn:exact function estimats2}) and $u=z$, we have $\|z\|_{H^1(\Omega)}\leq \alpha^{-1}L$. Thus, for the term $I_5$, utilizing Lemma \ref{eqn:lemma3}, we have
\begin{align}\label{error estimat14}
  I_5\leq \|\frac{\alpha}{2}u+p\|_{H^1(\Omega_h)}\|\tilde{z}_h-z\|_{H^{-1}(\Omega_h)}\leq c_2 L \|z\|_{H^1(\Omega_h)}h^2 \leq c_2 L^2 \alpha^{-1}h^2.
\end{align}
For terms $I_6$ and $I_7$, using H$\mathrm{\ddot{o}}$lder's inequality, Lemma \ref{eqn:lemma1} and Lemma \ref{eqn:lemma3}, we have
\begin{align}\label{error estimat15}
  I_6&\leq \frac{\alpha}{4}\|u_h-u\|^2_{L^2(\Omega_h)}+\frac{\alpha}{4}\|\tilde{z}_h-z\|^2_{L^2(\Omega_h)}\leq \frac{\alpha}{4}\|u_h-u\|^2_{L^2(\Omega_h)}+\frac{c_2^2L^2\alpha^{-1}}{4}h^2,
\end{align}
and
\begin{equation}\label{error estimat16}
  \begin{aligned}
  I_7&\leq\frac{1}{2}\|y-y_h\|^2_{L^2(\Omega_h)}+2\|\mathcal{S}_h-\mathcal{S}\|^2_{\mathcal{L}(L^{2},L^2)}
  (\|\tilde{z}_h\|^2_{L^2(\Omega_h)}+\|y_c\|^2_{L^2(\Omega_h)})
+\|\mathcal{S}\|_{\mathcal{L}(H^{-1},L^2)}\|z-\tilde{z}_h\|^2_{H^{-1}(\Omega_h)}\\
  &\leq\frac{1}{2}\|y-y_h\|^2_{L^2(\Omega_h)}+2c_1^2L^2h^4+c_2^2L^3\alpha^{-2}h^4.
\end{aligned}
\end{equation}
Finally, about the term $I_8$, we have
\begin{equation}\label{error estimat17}
  \begin{aligned}
  I_8&\leq \|y-y_d\|_{L^2(\Omega_h)}\|\mathcal{S}_h-\mathcal{S}\|_{\mathcal{L}(L^{2},L^2)}(\|\tilde{z}_h-z\|_{L^2(\Omega_h)}+\|z-z_h\|_{L^2(\Omega_h)})\\
  &\leq c_1Lh^2(c_2L\alpha^{-1}h+\|z-z_h\|_{L^2(\Omega_h)})\\
  &\leq \frac{\alpha}{4}\|z-z_h\|^2_{L^2(\Omega_h)}+c_1c_2\alpha^{-1}L^2h^3+4c_1^2L^2\alpha^{-1}h^4.
\end{aligned}
\end{equation}
Substituting (\ref{error estimat14}), (\ref{error estimat15}), (\ref{error estimat16}) and (\ref{error estimat17}) into (\ref{error estimat13}) and rearranging, we get
\begin{align*}
  \frac{\alpha}{2}\|u-u_h\|^2_{L^2(\Omega_h)}+\frac{1}{2}\|y-y_h\|^2_{L^2(\Omega_h)}\leq C(h^2+\alpha^{-1} h^2+\alpha^{-1}h^3+\alpha^{-1}h^4+\alpha^{-2}h^4),
\end{align*}
where $C>0$ is a properly chosen constant. Using again the assumption $|\Omega\backslash\Omega_h|\leq ch^2$, we can get
\begin{align*}
  \frac{\alpha}{2}\|u-u_h\|^2_{L^2(\Omega)}+\frac{1}{2}\|y-y_h\|^2_{L^2(\Omega)}\leq C(h^2+\alpha h^2+\alpha^{-1} h^2+h^3+\alpha^{-1}h^4+\alpha^{-2}h^4).
\end{align*}
\endproof

\begin{corollary}\label{corollary:error1}
Let $(y, u, z)$ be the optimal solution of problem {\rm(\ref{eqn:modified problems})}, and $(y_h, u_h, z_h)$ be the optimal solution of problem {\rm(\ref{equ:approx discretized})}. For every $h_0>0$, $\alpha_0>0$, there is a constant $C>0$ such that for all $0<\alpha\leq\alpha_0$, $0<h\leq h_0$ it holds
\begin{eqnarray*}
  \|u-u_h\|_{L^2(\Omega)}\leq C(\alpha^{-1}h+\alpha^{-\frac{3}{2}}h^2),
\end{eqnarray*}
where $C$ is a constant independent of $h$ and $\alpha$.
\end{corollary}

\section{An ihADMM algorithm and two-phase strategy for discretized problems}\label{sec:4} 
In this section, we will introduce an inexact ADMM algorithm and a two-phase strategy for discrete problems. Firstly, in order to establish relations parallel to (\ref{equ:saddle point problems2}) and (\ref{z-closed form solution}) for the discrete problem (\ref{equ:approx discretized}), we propose an inexact heterogeneous ADMM (ihADMM) algorithm with the aim of solving (\ref{equ:approx discretized}) to moderate accuracy.
Furthermore, as we have mentioned, if more accurate solution is necessarily required, combining our ihADMM and the primal-dual active set (PDAS) method is a wise choice. Then a two-phase strategy is introduced. Specifically, utilizing the solution generated by our ihADMM, as a reasonably good initial point, the PDAS method is used as a postprocessor of our ihADMM. 

Firstly, let us define following stiffness and mass matrices:
\begin{eqnarray*}
K_h &=& \left(a(\phi_i, \phi_j)\right)_{i,j=1}^{N_h},\quad
M_h=\left(\int_{\Omega_h}\phi_i\phi_j{\mathrm{d}}x\right)_{i,j=1}^{N_h},
\end{eqnarray*}
where the bilinear form $a(\cdot,\cdot)$ is defined in (\ref{eqn:bilinear form}).

Due to the quadrature formulas (\ref{eqn:approx norm2}) and (\ref{eqn:approx norm1}), a lumped mass matrix
 $ W_h={\rm{diag}}\left(\int_{\Omega_h}\phi_i(x)\mathrm{d}x\right)_{i,j=1}^{N_h}$
is introduced. Moreover, by (\ref{equ:martix properties1}) in Proposition \ref{eqn:martix properties}, we have the following results about the mass matrix $M_h$ and the lump mass matrix $W_h$.

\subsection{An inexact heterogeneous ADMM algorithm}

Denoting by $y_{d,h}:=\sum\limits_{i=1}^{N_h}y_d^i\phi_i(x)$ and $y_{c,h}:=\sum\limits_{i=1}^{N_h}y_c^i\phi_i(x)$ the $L^2$-projection of $y_d$ and $y_c$ onto $Y_h$, respectively, and identifying discretized functions with their coefficient vectors, we can rewrite the problem (\ref{equ:approx discretized}) as a matrix-vector form:
\begin{equation}\label{equ:approx discretized matrix-vector form}
\left\{\begin{aligned}
        &\min\limits_{(y,u,z)\in\mathbb{R}^{3N_h}}^{}~~ \frac{1}{2}\|y-y_d\|_{M_h}^{2}+\frac{\alpha}{4}\|u\|_{M_h}^{2}+\frac{\alpha}{4}\|z\|_{W_h}^{2}+\|W_hz\|_1\\
        &~~~ \quad {\rm{s.t.}}\qquad\quad K_hy=M_h(u+y_c),\\
        &\ \qquad\quad\quad\quad\quad u=z,\\
        &\ \qquad\quad\quad\quad\quad z\in[a,b]^{N_h}.
                          \end{aligned} \right.\tag{$\overline{\mathrm{DP}}_{h}$}
\end{equation}
By Assumption \ref{equ:assumption:1}, we have the stiffness matrix $K_h$ is a symmetric positive definite matrix. Then problem (\ref{equ:approx discretized matrix-vector form}) can be rewritten the following reduced form:
\begin{equation}\label{equ:reduced approx discretized matrix-vector form}
\left\{\begin{aligned}
        &\min\limits_{(u,z)\in\mathbb{R}^{2N_h}}^{}~~ f(u)+g(z)\\
        &~~ \quad {\rm{s.t.}} \quad \qquad u=z.
                          \end{aligned} \right.\tag{$\overline{\mathrm{RDP}}_{h}$}
\end{equation}
where
\begin{eqnarray}
 \label{equ:fu function} f(u)&=& \frac{1}{2}\|K_h^{-1}M_h(u+y_c)-y_d\|_{M_h}^{2}+\frac{\alpha}{4}\|u\|_{M_h}^{2}, \quad
g(z) = \frac{\alpha}{4}\|z\|_{W_h}^{2}+\beta\|W_hz\|_1+\delta_{[a,b]^{N_h}}.
\end{eqnarray}

To solve (\ref{equ:reduced approx discretized matrix-vector form}) by using ADMM-type algorithm, we first introduce the augmented Lagrangian function for (\ref{equ:reduced approx discretized matrix-vector form}). According to three possible choices of norms ($\mathbb{R}^{N_h}$ norm, $W_h$-weighted norm and $M_h$-weighted norm), for the augmented Lagrangian function, there are three versions as follows: for given $\sigma>0$,
\begin{eqnarray}
\mathcal{L}^1_\sigma(u,z;\lambda)&:=&f(u)+g(z)+\langle\lambda,u-z\rangle+\frac{\sigma}{2}\|u-z\|^{2}, \label{aguLarg1}\\
\mathcal{L}^2_\sigma(u,z;\lambda)&:=&f(u)+g(z)+\langle\lambda,M_h(u-z)\rangle+\frac{\sigma}{2}\|u-z\|_{W_h}^{2}, \label{aguLarg2}\\
\mathcal{L}^3_\sigma(u,z;\lambda)&:=&f(u)+g(z)+\langle\lambda,M_h(u-z)\rangle+\frac{\sigma}{2}\|u-z\|_{M_h}^{2}\label{aguLarg3}.
\end{eqnarray}
Then based on these three versions of augmented Lagrangian function, we give the following four versions of ADMM-type algorithm for (\ref{equ:reduced approx discretized matrix-vector form}) at $k$-th ineration: for given $\tau>0$ and $\sigma>0$,

\begin{equation}\label{inexact ADMM1}
  \left\{\begin{aligned}
  &u^{k+1}=\arg\min_u\ f(u)+\langle\lambda^k,u-z^k\rangle+\sigma/2\|u-z^k\|^{2},\\
  &z^{k+1}=\arg\min_z\ g(z)+\langle\lambda^k,u^{k+1}-z\rangle+\sigma/2\|u^{k+1}-z\|^{2},\\
  &\lambda^{k+1}=\lambda^k+\tau\sigma(u^{k+1}-z^{k+1}).
  \end{aligned}\right.\tag{ADMM1}
\end{equation}
\begin{equation}\label{inexact ADMM2}
  ~~~~~~~~ \left\{\begin{aligned}
  &u^{k+1}=\arg\min_u\ f(u)+\langle\lambda^k,M_h(u-z^k)\rangle+\sigma/2\|u-z^k\|_{W_h}^{2},\\
  &z^{k+1}=\arg\min_z\ g(z)+\langle\lambda^k,W_h(u^{k+1}-z)\rangle+\sigma/2\|u^{k+1}-z\|_{W_h}^{2},\\
  &\lambda^{k+1}=\lambda^k+\tau\sigma(u^{k+1}-z^{k+1}).
  \end{aligned}\right.\tag{ADMM2}
\end{equation}
\begin{equation}\label{inexact ADMM3}
  ~~~~~~~~ \left\{\begin{aligned}
  &u^{k+1}=\arg\min_u\ f(u)+\langle\lambda^k,M_h(u-z^k)\rangle+\sigma/2\|u-z^k\|_{M_h}^{2},\\
  &z^{k+1}=\arg\min_z\ g(z)+\langle\lambda^k,M_h(u^{k+1}-z)\rangle+\sigma/2\|u^{k+1}-z\|_{M_h}^{2},\\
  &\lambda^{k+1}=\lambda^k+\tau\sigma(u^{k+1}-z^{k+1}).
  \end{aligned}\right.\tag{ADMM3}
\end{equation}
\begin{equation}\label{inexact ADMM4}
  ~~~~~~~~ \left\{\begin{aligned}
  &u^{k+1}=\arg\min_u\ f(u)+\langle\lambda^k,M_h(u-z^k)\rangle+{\color{blue}\sigma/2\|u-z^k\|_{M_h}^{2}},\\
  &z^{k+1}=\arg\min_z\ g(z)+\langle\lambda^k,M_h(u^{k+1}-z)\rangle+{\color{red}\sigma/2\|u^{k+1}-z\|_{W_h}^{2}},\\
  &\lambda^{k+1}=\lambda^k+\tau\sigma(u^{k+1}-z^{k+1}).
  \end{aligned}\right.\tag{ADMM4}
\end{equation}

As one may know, (\ref{inexact ADMM1}) is actually the classical ADMM for (\ref{equ:reduced approx discretized matrix-vector form}). The remaining three ADMM-type algorithms are proposed based on the structure of (\ref{equ:reduced approx discretized matrix-vector form}). Now, let us start to analyze and compare the advantages and disadvantages of the four algorithms.
Firstly, we focus on the $z$-subproblem in each algorithm. Since both identity matrix $I$ and lumped mass matrix $W_h$ are diagonal, it is clear that all the $z$-subproblems in (\ref{inexact ADMM1}), (\ref{inexact ADMM2}) and (\ref{inexact ADMM4}) have a closed form solution, except for the $z$-subproblem in (\ref{inexact ADMM3}).
Specifically, for $z$-subproblem in (\ref{inexact ADMM1}), the closed form solution could be given by:
\begin{equation}\label{equ:closed form solution for ADMM1}
z^k={\rm\Pi}_{U_{ad}}\left((\frac{\alpha}{2}W_h+\sigma I)^{-1}W_h{\rm soft}(W_h^{-1}(\sigma u^{k+1}+\lambda^k),\beta)\right).
\end{equation}
Similarly, for $z$-subproblems in (\ref{inexact ADMM2}) and (\ref{inexact ADMM4}), the closed form solution could be given by:
\begin{equation}\label{equ:closed form solution for ADMM2 and ADMM4}
z^{k+1}={\rm\Pi}_{U_{ad}}\left(\frac{1}{\sigma+0.5\alpha}{\rm soft}\left(\sigma u^{k+1}+W_h^{-1}M_h\lambda^k{\text{,}}~\beta\right)\right)
\end{equation}
Fortunately, the expression of (\ref{equ:closed form solution for ADMM2 and ADMM4}) is the similar to (\ref{z-closed form solution}). As we have mentioned that, from the view of both the actual numerical implementation and convergence analysis of the algorithm, establishing such parallel relation is important.


Next, let us analyze the structure of $u$-subproblem in each algorithm. For (\ref{inexact ADMM1}), the first subproblem at $k$-th iteration 
is equivalent to solving the following linear system:
\begin{equation}\label{eqn:saddle point1}
\left[
  \begin{array}{ccc}
    M_h & \quad0 & \quad K_h \\
    0 & \quad\frac{\alpha}{2}M_h+\sigma I & \quad-M_h \\
    K_h & \quad-M_h & \quad0 \\
  \end{array}
\right]\left[
         \begin{array}{c}
           y^{k+1} \\
           u^{k+1} \\
           p^{k+1} \\
         \end{array}
       \right]=\left[
                 \begin{array}{c}
                   M_hy_d \\
                   \sigma z^k-\lambda^k \\
                   M_hy_c \\
                 \end{array}
               \right].
\end{equation}
Similarly, the $u$-subproblem in (\ref{inexact ADMM2}) can be converted into the following linear system:
\begin{equation}\label{eqn:saddle point2}
\left[
  \begin{array}{ccc}
    M_h & \quad0 & \quad K_h \\
    0 & \quad\frac{\alpha}{2}M_h+\sigma W_h & \quad-M_h \\
    K_h & \quad-M_h & \quad0 \\
  \end{array}
\right]\left[
         \begin{array}{c}
           y^{k+1} \\
           u^{k+1} \\
           p^{k+1} \\
         \end{array}
       \right]=\left[
                 \begin{array}{c}
                   M_hy_d \\
                   \sigma W_h(z^k-\lambda^k) \\
                   M_hy_c \\
                 \end{array}
               \right].
\end{equation}

However, the $u$-subproblem in both (\ref{inexact ADMM3}) and (\ref{inexact ADMM4}) can be rewritten as:
\begin{equation}\label{eqn:saddle point3}
\left[
  \begin{array}{ccc}
    M_h & \quad0 & \quad K_h \\
    0 & \quad(0.5\alpha+\sigma) M_h  & \quad-M_h \\
    K_h & \quad-M_h & \quad0 \\
  \end{array}
\right]\left[
         \begin{array}{c}
           y^{k+1} \\
           u^{k+1} \\
           p^{k+1} \\
         \end{array}
       \right]=\left[
                 \begin{array}{c}
                   M_hy_d \\
                   M_h(\sigma z^k-\lambda^k) \\
                   M_hy_c \\
                 \end{array}
               \right].
\end{equation}
In (\ref{eqn:saddle point3}), since $p^{k+1}=(0.5\alpha+\sigma)u^{k+1}-\sigma z^k+\lambda^k$, it is obvious that (\ref{eqn:saddle point3}) can be reduced into the following system by eliminating the variable $p$ without any computational cost:
\begin{equation}\label{eqn:saddle point4}
\left[
  \begin{array}{cc}
    \frac{1}{0.5\alpha+\sigma}M_h & K_h \\
    -K_h & M_h
  \end{array}
\right]\left[
         \begin{array}{c}
           y^{k+1} \\
           u^{k+1}
         \end{array}
       \right]=\left[
                 \begin{array}{c}
                   \frac{1}{0.5\alpha+\sigma}(K_h(\sigma z^k-\lambda^k)+M_hy_d)\\
                   -M_hy_c
                 \end{array}
               \right],
\end{equation}
while, reduced forms of (\ref{eqn:saddle point1}) and (\ref{eqn:saddle point2}):
both involve the inversion of $M_h$.

For above mentioned reasons, we prefer to use (\ref{inexact ADMM4}), which is called the heterogeneous ADMM (hADMM). However, in general, it is expensive and unnecessary to exactly compute the solution of saddle point system (\ref{eqn:saddle point4}) even if it is doable, especially at the early stage of the whole process. Based on the structure of (\ref{eqn:saddle point4}), it is a natural idea to use the iterative methods such as some Krylov-based methods. Hence, taking the inexactness of the solution of $u$-subproblem into account, a more practical inexact heterogeneous ADMM (ihADMM) algorithm is proposed.

Due to the inexactness of the proposed algorithm, we first introduce an error tolerance. Throughout this paper, let $\{\epsilon_k\}$ be a summable sequence of nonnegative numbers, and define
\begin{equation}\label{error sequence}
  C_1:=\sum\limits^{\infty}_{k=0}\epsilon_k\leq\infty, \quad C_2:=\sum\limits^{\infty}_{k=0}\epsilon_k^2\leq\infty.
\end{equation}
The details of our ihADMM algorithm is shown in Algorithm \ref{algo4:inexact heterogeneous ADMM for problem RHP} to solve (\ref{equ:approx discretized matrix-vector form}).

\begin{algorithm}[H]
  \caption{inexact heterogeneous ADMM algorithm for (\ref{equ:approx discretized matrix-vector form})}\label{algo4:inexact heterogeneous ADMM for problem RHP}
  \textbf{Input}: {$(z^0, u^0, \lambda^0)\in {\rm dom} (\delta_{[a,b]}(\cdot))\times \mathbb{R}^n \times \mathbb{R}^n $ and parameters $\sigma>0$, $\tau>0$. 
  Set $k=1$.}\\
  \textbf{Output}: {$ u^k, z^{k}, \lambda^k$}
\begin{description}
\item[Step 1] Find an minizer (inexact)
\begin{eqnarray*}
  u^{k+1}&=&\arg\min f(u)+(M_h\lambda^k,u-z^k)
         +\frac{\sigma}{2}\|u-z^k\|_{M_h}^{2}-\langle\delta^k, u\rangle,
\end{eqnarray*}
where the error vector ${\delta}^k$ satisfies $\|{\delta}^k\|_{2} \leq {\epsilon_k}$
\item[Step 2] Compute $z^k$ as follows:
       \begin{eqnarray*}
       z^{k+1}&=&\arg\min g(z)+(M_h\lambda^k,u^{k+1}-z)
         +\frac{\sigma}{2}\|u^{k+1}-z\|_{W_h}^{2}
       \end{eqnarray*}
  \item[Step 3] Compute
  \begin{eqnarray*}
    \lambda^{k+1} &=& \lambda^k+\tau\sigma(u^{k+1}-z^{k+1}).
  \end{eqnarray*}

  \item[Step 4] If a termination criterion is not met, set $k:=k+1$ and go to Step 1
\end{description}
\end{algorithm}
\subsection{Convergence results of ihADMM}
For the ihADMM (Algorithm \ref{algo4:inexact heterogeneous ADMM for problem RHP}), in this section we establish the global convergence and the iteration complexity results in non-ergodic sense for the sequence generated by Algorithm \ref{algo4:inexact heterogeneous ADMM for problem RHP}.

Before giving the proof of Theorem \ref{discrete convergence results}, we first provide a lemma, which is useful for analyzing the non-ergodic iteration complexity of ihADMM and introduced in \cite{SunToh1}.
\begin{lemma}\label{complexity lemma}
If a sequence $\{a_i\}\in \mathbb{R}$ satisfies the following conditions:
\begin{eqnarray*}
  &&a_i\geq0 \ \text{for any}\ i\geq0\quad and\quad \sum\limits_{i=0}^{\infty}a_i=\bar a<\infty.
\end{eqnarray*}
Then we have $\min\limits^{}_{i=1,...,k}\{a_i\} \leq \frac{\bar a}{k}$, and $\lim\limits^{}_{k\rightarrow\infty} \{k\cdot\min\limits^{}_{i=1,...,k}\{a_i\}\} =0$.
\end{lemma}

For the convenience of the iteration complexity analysis in below, we define the function $R_h: (u,z,\lambda)\rightarrow [0,\infty)$ by:
\begin{equation}\label{discrete KKT function}
  R_h(u,z,\lambda)=\|M_h\lambda+\nabla f(u)\|^2+{\rm dist}^2(0, -M_h\lambda+\partial g(z))+\|u-z\|^2.
\end{equation}

By the definitions of $f(u)$ and $g(z)$ in (\ref{equ:fu function}), it is obvious that $f(u)$ and $g(z)$ both are closed, proper and convex functions. Since $M_h$ and $K_h$ are symmetric positive definite matrixes, we know the gradient operator $\nabla f$ is strongly monotone, and we have
 \begin{equation}\label{subdifferential strongly monotone}
   \langle\nabla f(u_1)-\nabla f(u_2), u_1-u_2\rangle=\|u_1-u_2\|^2_{\Sigma_{f}},
   \end{equation}
where ${\Sigma_{f}}=\frac{\alpha}{2} M_h+M_hK_h^{-1}M_hK_h^{-1}M_h$ is symmetric positive definite. Moreover, the subdifferential operator $\partial g$ is a maximal monotone operators, e.g.,
 \begin{equation}\label{subdifferential monotone}
   \langle\varphi_1-\varphi_2, z_1-z_2\rangle\geq\frac{\alpha}{2}\|z_1-z_2\|^2_{ W_h} \quad \forall\ \varphi_1\in\partial g(z_1),\ \varphi_2\in \partial g(z_2).
 \end{equation}
For the subsequent convergence analysis, we denote
\begin{eqnarray}
  \label{equ:exact u}\bar{u}^{k+1}&:=&\arg\min f(u)+\langle M_h\lambda^k,u-z^k\rangle
         +\frac{\sigma}{2}\|u-z^k\|_{M_h}^{2},\\
  \label{equ:exact z} \bar z^{k+1}&:=&{\rm\Pi}_{U_{ad}}\left(\frac{1}{\sigma+0.5\alpha}{\rm soft}\left(\sigma \bar u^{k+1}+W_h^{-1}M_h\lambda^k{\text{,}}~\beta\right)\right),
\end{eqnarray}
which are the exact solutions at the $(k+1)$-th iteration in Algorithm \ref{algo4:inexact heterogeneous ADMM for problem RHP}. The following results show the gap between $(u^{k+1}, z^{k+1})$ and $(\bar u^{k+1}, \bar z^{k+1})$ in terms of the given error tolerance $\|{\delta}^k\|_{2} \leq {\epsilon_k}$.

\begin{lemma}\label{gap between exact and inexact solution}
  Let $\{(u^{k+1}, z^{k+1})\}$ be the squence generated by Algorithm {\rm\ref{algo4:inexact heterogeneous ADMM for problem RHP}}, and $\{\bar u^{k+1}\}$, $\{\bar z^{k+1}\}$ be defined in {\rm(\ref{equ:exact u})} and {\rm(\ref{equ:exact z})}. Then for any $k\geq0$, we have
  \begin{eqnarray}
    \label{equ:error u}\|u^{k+1}-\bar u^{k+1}\| &=&\|(\sigma M_h+\Sigma_f)^{-1}\delta^k\|\leq \rho\epsilon_k,  \\
    \label{equ:error z}\|z^{k+1}-\bar z^{k+1}\| &\leq&\|u^{k+1}-\bar u^{k+1}\|\leq \frac{\rho\sigma}{\sigma+0.5\alpha}\epsilon_k,
  \end{eqnarray}
where $\rho:=\|(\sigma M_h+\Sigma_f)^{-1}\|$.
\end{lemma}

Next, for $k\geq0$, we define
\begin{eqnarray*}
  &r^k=u^k-z^k,\quad \bar r^k=\bar u^k-\bar z^k&  \\
  &\tilde{\lambda }^{k+1}=\lambda^k+\sigma r^{k+1},\quad \bar{\lambda }^{k+1}=\lambda^k+\tau\sigma \bar r^{k+1}, \quad \hat{\lambda }^{k+1}=\lambda^k+\sigma \bar r^{k+1},&
\end{eqnarray*}
and give two inequalities which is essential for establishing both the global convergence and the iteration complexity of our ihADMM

\begin{proposition}\label{descent proposition}
Let $\{(u^{k}, z^{k}, \lambda^{k})\}$ be the sequence generated by Algorithm {\rm\ref{algo4:inexact heterogeneous ADMM for problem RHP}} and $(u^{*}, z^{*}, \lambda^{*})$ be the KKT point of problem {\rm(\ref{equ:reduced approx discretized matrix-vector form})}. Then for $k\geq0$ we have
\begin{equation}\label{inequlaities property1}
  \begin{aligned}
  &\langle\delta^k,u^{k+1}-u^*\rangle +\frac{1}{2\tau\sigma}\|\lambda^k-\lambda^*\|^2_{M_h}+\frac{\sigma}{2}\|z^k-z^*\|^2_{M_h}
  -\frac{1}{2\tau\sigma}\|\lambda^{k+1}-\lambda^*\|^2_{M_h}-\frac{\sigma}{2}\|z^{k+1}-z^*\|^2_{M_h}\\
  &\geq\|u^{k+1}-u^*\|^2_{T}
  +\frac{\sigma}{2}\|z^{k+1}-z^*\|^2_{2W_h-M_h} +\frac{\sigma}{2}\|r^{k+1}\|^2_{W_h-\tau M_h}
  +\frac{\sigma}{2}\|u^{k+1}-z^k\|^2_{M_h},
  \end{aligned}
\end{equation}
where $T:=\Sigma_f-\frac{\sigma}{2}(W_h-M_h)$.
\end{proposition}
\begin{proposition}\label{descent proposition2}
Let $\{(u^{k}, z^{k}, \lambda^{k})\}$ be the sequence generated by Algorithm {\rm\ref{algo4:inexact heterogeneous ADMM for problem RHP}}, $(u^{*}, z^{*}, \lambda^{*})$ be the KKT point of the problem {\rm(\ref{equ:reduced approx discretized matrix-vector form})} and $\{\bar u^k\}$ and $\{\bar z^k\}$ be two sequences defined in {\rm(\ref{equ:exact u})} and {\rm(\ref{equ:exact z})}, respectively. Then for $k\geq0$ we have
\begin{equation}\label{inequlaities property}
  \begin{aligned}
 &\frac{1}{2\tau\sigma}\|\lambda^k-\lambda^*\|^2_{M_h}+\frac{\sigma}{2}\|z^k-z^*\|^2_{M_h}
  -\frac{1}{2\tau\sigma}\|\bar \lambda^{k+1}-\lambda^*\|^2_{M_h}-\frac{\sigma}{2}\|\bar z^{k+1}-z^*\|^2_{M_h}\\
  \geq\ &\|\bar u^{k+1}-u^*\|^2_{T}
  +\frac{\sigma}{2}\|\bar z^{k+1}-z^*\|^2_{2W_h-M_h} +\frac{\sigma}{2}\|\bar r^{k+1}\|^2_{W_h-\tau M_h}+\frac{\sigma}{2}\|\bar u^{k+1}-z^k\|^2_{M_h},
  \end{aligned}
\end{equation}
where $T:=\Sigma_f-\frac{\sigma}{2}(W_h-M_h)$.
\end{proposition}
Then based on former results, we have the following convergence results.
\begin{theorem}\label{discrete convergence results}
Let $(y^*,u^*,z^*,p^*,\lambda^*)$ is the KKT point of {\rm(\ref{equ:approx discretized matrix-vector form})}, then the sequence $\{(u^{k},z^{k},\lambda^k)\}$ is generated by Algorithm {\rm\ref{algo4:inexact heterogeneous ADMM for problem RHP}} with the associated state $\{y^k\}$ and adjoint state $\{p^k\}$, then for any $\tau\in (0,1]$ and $\sigma\in (0, \frac{1}{4}\alpha]$, we have
\begin{eqnarray}
  \label{discrete iteration squence convergence1}&&\lim\limits_{k\rightarrow\infty}^{}\{\|u^{k}-u^*\|+\|z^{k}-z^*\|+\|\lambda^{k}-\lambda^*\| \}= 0\\
   \label{discrete iteration squence convergence2}&& \lim\limits_{k\rightarrow\infty}^{}\{\|y^{k}-y^*\|+\|p^{k}-p^*\| \}= 0
\end{eqnarray}
Moreover, there exists a constant $C$ only depending on the initial point ${(u^0,z^0,\lambda^0)}$ and the optimal solution ${(u^*,z^*,\lambda^*)}$ such that for $k\geq1$,
\begin{eqnarray}
  \label{discrete iteration complexity1}&&\min\limits^{}_{1\leq i\leq k} \{R_h(u^i,z^i,\lambda^i)\}\leq\frac{C}{k}, \quad
\lim\limits^{}_{k\rightarrow\infty}\left(k\times\min\limits^{}_{1\leq i\leq k} \{R_h(u^i,z^i,\lambda^i)\}\right) =0.
\end{eqnarray}
where $R_h(\cdot)$ is defined as in {\rm(\ref{discrete KKT function})}.
%
\begin{proof}
It is easy to see that $(u^*,z^*)$ is the unique optimal solution of discrete problem (\ref{equ:reduced approx discretized matrix-vector form}) if and only if there exists a Lagrangian multiplier $\lambda^*$ such that the following Karush-Kuhn-Tucker (KKT) conditions hold,
\begin{subequations}
\begin{eqnarray}
  \label{equ: exact variational inequalities1}&-M_h\lambda^*=\nabla f(u^*),\\
  \label{equ: exact variational inequalities2}&M_h\lambda^*\in \partial g(z^*),\\
  \label{equ: exact variational inequalities3}& u^*=z^*.
\end{eqnarray}
\end{subequations}

In the inexact heterogeneous ADMM iteration scheme, the optimality conditions for $(u^{k+1}, z^{k+1})$ are
\begin{subequations}
\begin{eqnarray}
  \label{equ: inexact variational inequalities1}&\delta^k-(M_h\lambda^k+\sigma M_h(u^{k+1}-z^k))=\nabla f(u^{k+1}),\\
  \label{equ: inexact variational inequalities2}&M_h\lambda^k+\sigma W_h(u^{k+1}-z^{k+1})\in \partial g(z^{k+1}).
\end{eqnarray}
\end{subequations}
Next, let us first prove the \textbf{global convergence of iteration sequences,} e.g., establish the proof of (\ref{discrete iteration squence convergence1}) and (\ref{discrete iteration squence convergence2}).

The first step is to show that $\{(u^k, z^k, \lambda^k)\}$ is bounded. We define the following sequence $\theta^k$ and $\bar\theta^k$ with:
\begin{equation}\label{iteration sequence}
  \begin{aligned}
  \theta^k &= \left(\frac{1}{\sqrt{2\tau\sigma}}M_h^{\frac{1}{2}}(\lambda^k-\lambda^*), \sqrt{\frac{\sigma}{2}}M_h^{\frac{1}{2}}(z^k-z^*)\right), \quad
   \bar\theta^k = \left(\frac{1}{\sqrt{2\tau\sigma}}M_h^{\frac{1}{2}}(\bar\lambda^k-\lambda^*), \sqrt{\frac{\sigma}{2}}M_h^{\frac{1}{2}}(\bar z^k-z^*)\right).
  \end{aligned}
\end{equation}
According to Proposition \ref{eqn:martix properties}, for any $\tau\in (0,1]$ and $\sigma\in (0, \frac{1}{4}\alpha]$ for, we have
  $\Sigma_f-\frac{\sigma}{2}(W_h-M_h) \succ 0$, and
  $\quad W_h-\tau M_h \succ 0$ .
Then, by Proposition \ref{descent proposition2}, we get $\|\bar\theta^{k+1}\|^2\leq\|\theta^k\|^2$. As a result, we have:
\begin{equation}\label{descent theta}
  \begin{aligned}
  \|\theta^{k+1}\| &\leq \|\bar\theta^{k+1}\|+\|\bar\theta^{k+1}-\theta^{k+1}\|
   = \|\bar\theta^{k}\|+\|\bar\theta^{k+1}-\theta^{k+1}\| .
  \end{aligned}
\end{equation}
Employing Lemma \ref{gap between exact and inexact solution}, we get
\begin{equation}\label{descent theta and bartheta}
  \begin{aligned}
  \|\bar\theta^{k+1}-\theta^{k+1}\|^2 &= \frac{1}{2\tau\sigma}\|\bar\lambda^{k+1}-\lambda^{k+1}\|^2_{M_h}+\frac{\sigma}{2}\|\bar z^{k+1}-z^{k+1}\|^2_{M_h} \\
  &\leq(2\tau+1/2)\sigma\|M_h\|\rho^2\epsilon_k^2\leq5/2\sigma\|M_h\|\rho^2\epsilon_k^2,
  \end{aligned}
\end{equation}
which implies $\|\bar\theta^{k+1}-\theta^{k+1}\|\leq\sqrt{5/2\sigma\|M_h\|}\rho\epsilon_k$. Hence, for any $k\geq0$, we have
\begin{equation}\label{boundedness theta and bartheta}
  \begin{aligned}
  \|\theta^{k+1}\| &\leq \|\theta^k\|+\sqrt{5/2\sigma\|M_h\|}\rho\epsilon_k
  \leq  \|\theta^0\|+\sqrt{5/2\sigma\|M_h\|}\rho\sum\limits^{\infty}_{k=0}\epsilon_k=\|\theta^0\|+\sqrt{5/2\sigma\|M_h\|}\rho C_1\equiv\bar\rho.
  \end{aligned}
\end{equation}
From $\|\bar\theta^{k+1}\|\leq\|\theta^{k}\|$, for any $k\geq0$, we also have $\|\bar\theta^{k+1}\|\leq\bar\rho$. Therefore, the sequences $\{\theta^k\}$ and $\{\bar \theta^k\}$ are bounded. From the definition of $\{\theta^k\}$ and the fact that $M_h\succ0$, we can see that the sequences $\{\lambda^k\}$ and $\{z^k\}$ are bounded. Moreover, from updating technique of $\lambda^k$, we know $\{u^k\}$ is also bounded. Thus, due to the boundedness of the sequence $\{(u^{k}, z^{k}, \lambda^k)\}$, we know the sequence has a subsequence
$\{(u^{k_i}, z^{k_i}, \lambda^{k_i})\}$ which converges to an accumulation point $(\bar u, \bar z, \bar\lambda)$. Next we should show that $(\bar u, \bar z, \bar\lambda)$ is a KKT point and equal to $(u^*, z^*, \lambda^*)$. 

Again employing Proposition \ref{descent proposition2}, we can derive
\begin{equation}\label{convergence inequlity}
  \begin{aligned}
&\sum\limits^{\infty}_{k=0}\left(\|\bar u^{k+1}-u^*\|^2_{T}
  +\frac{\sigma}{2}\|\bar z^{k+1}-z^*\|^2_{2W_h-M_h} +\frac{\sigma}{2}\|\bar r^{k+1}\|^2_{W_h-\tau M_h}+\frac{\sigma}{2}\|\bar u^{k+1}-z^k\|^2_{M_h}\right)\\
  \leq &\sum\limits^{\infty}_{k=0}(\|\theta^k\|^2-\|\theta^{k+1}\|^2+\|\theta^{k+1}\|^2-\|\bar \theta^{k+1}\|^2)
  \leq \|\theta^0\|^2+2\bar\rho\sqrt{5/2\sigma\|M_h\|}\rho C_1<\infty.
  \end{aligned}
\end{equation}
Note that $T\succ0, W_h-M_h\succ0, W_h-\tau M_h\succ0$ and $M_h\succ0$, then we have
\begin{equation}\label{limit convergence2}
  \begin{aligned}
  \lim\limits^{}_{k\rightarrow\infty}\|\bar u^{k+1}-u^*\|=0,\quad \lim\limits^{}_{k\rightarrow\infty}\|\bar z^{k+1}-z^*\|=0,\quad
  \lim\limits^{}_{k\rightarrow\infty} \|\bar r^{k+1}\|=0,\quad
  \lim\limits^{}_{k\rightarrow\infty}\|\bar u^{k+1}-z^k\|=0.
    \end{aligned}
\end{equation}
From the Lemma \ref{gap between exact and inexact solution}, we can get
\begin{equation}\label{limit convergence3}
  \begin{aligned}
  &\|u^{k+1}-u^*\|\leq\|\bar u^{k+1}-u^*\|+\|u^{k+1}-\bar u^{k+1}\|\leq\|\bar u^{k+1}-u^*\|+\rho\epsilon_k,\\
  &\|z^{k+1}-z^*\|\leq\|\bar z^{k+1}-z^*\|+\|z^{k+1}-\bar z^{k+1}\|\leq\|\bar z^{k+1}-z^*\|+\rho\epsilon_k.
    \end{aligned}
\end{equation}
From the fact that $\lim\limits^{}_{k\rightarrow\infty} \epsilon_k=0$ and (\ref{limit convergence2}), by taking the limit of both sides of (\ref{limit convergence3}), we have
\begin{equation}\label{limit convergence4}
  \begin{aligned}
  \lim\limits^{}_{k\rightarrow\infty}\| u^{k+1}-u^*\|=0,\quad \lim\limits^{}_{k\rightarrow\infty}\| z^{k+1}-z^*\|=0,\quad
  \lim\limits^{}_{k\rightarrow\infty} \| r^{k+1}\|=0,\quad
  \lim\limits^{}_{k\rightarrow\infty}\| u^{k+1}-z^k\|=0.
    \end{aligned}
\end{equation}
Now taking limits for $k_i\rightarrow\infty$ on both sides of (\ref{equ: inexact variational inequalities1}), we have
\begin{equation*}
  \lim\limits^{}_{k_i\rightarrow\infty}(\delta^{k_i}-(M_h\lambda^{k_i}+\sigma M_h(u^{k_i+1}-z^{k_i})))=\nabla f(u^{k_i+1}),
\end{equation*}
which results in $-M_h\bar\lambda=\nabla f(u^*)$. Then from (\ref{equ: exact variational inequalities1}), we know $\bar\lambda=\lambda^*$. At last, to complete the proof, we need to show that $\lambda^*$ is the limit of the sequence of $\{\lambda^k\}$. From
(\ref{boundedness theta and bartheta}), we have for any $k>k_i$,
  $\|\theta^{k+1}\|\leq\|\theta^{k_i}\|+\sqrt{5/2\sigma\|M_h\|}\rho\sum\limits^{k}_{j={k_i}}\epsilon_j$.
Since $\lim\limits^{}_{k_i\rightarrow\infty}\|\theta^{k_i}\|=0$ and $\sum\limits_{k=0}^{\infty}\epsilon_k<\infty$, we have that $\lim\limits^{}_{k\rightarrow\infty}\|\theta^{k}\|=0$, which implies $\lim\limits^{}_{k\rightarrow\infty}\| \lambda^{k+1}-\lambda^*\|=0$.
Hence, we have proved the convergence of the sequence $\{(u^{k+1}, z^{k+1}, \lambda^{k+1})\}$, which completes the proof of (\ref{discrete iteration squence convergence1}). For the proof of (\ref{discrete iteration squence convergence2}), it is easy to show by the definition of the sequence $\{(y^k, p^k)\}$, here we omit it.

At last, we establish the proof of (\ref{discrete iteration complexity1}), e.g., \textbf{the iteration complexity results in non-ergodic sendse for the sequence generated by the ihADMM.}

Firstly, by the optimality condition (\ref{equ: inexact variational inequalities1}) and (\ref{equ: inexact variational inequalities2}) for $(u^{k+1}, z^{k+1})$, we have
\begin{subequations}
\begin{eqnarray}
 \label{equ:discrete inexact variational inequalities3}&\delta^k+(\tau-1)\sigma M_h r^{k+1}-\sigma M_h(z^{k+1}-z^k)=M_h\lambda^{k+1}+\nabla f(u^{k+1}),\\
  \label{equ:discrete inexact variational inequalities4}&\sigma (W_h-\tau M_h)r^{k+1}\in -M_h\lambda^{k+1}+\partial g(z^{k+1}).
\end{eqnarray}
\end{subequations}
By the definition of $R_h$ and denoting $w^{k+1}:=(u^{k+1},z^{k+1},\lambda^{k+1})$, we derive
\begin{equation}\label{discrete KKT function for iteration}
\begin{aligned}
  R_h(w^{k+1})&=\|M_h\lambda^{k+1}+\nabla f(u^{k+1})\|^2+{\rm dist}^2(0, -M_h\lambda^{k+1}+\partial g(z^{k+1}))+\|u^{k+1}-z^{k+1}\|^2\\
  &\leq 2\|\delta^k\|^2+\eta\|r^{k+1}\|^2+4\sigma^2 \|M_h\|\|u^{k+1}-z^{k}\|_{M_h}^2,
\end{aligned}
\end{equation}
where $\eta:=2(\tau-1)^2\sigma^2 \|M_h\|^2+2\sigma^2 \|M_h\|^2+\sigma^2\| W_h-\tau M_h\|^2+1.$

In order to get a upper bound for $R_h(w^{k+1})$, we will use (\ref{inequlaities property1}) in Proposition
\ref{descent proposition}. First, by the definition of $\theta^{k}$ and (\ref{boundedness theta and bartheta}), for any $k\geq0$ we can easily have
\begin{eqnarray*}
  \|\lambda^k-\lambda^*\| \leq\bar\rho\sqrt{\frac{2\tau\sigma}{\|M_h^{-1}\|}} , \quad
  \|z^k-z^*\| \leq\bar\rho\sqrt{\frac{2}{\sigma\|M_h^{-1}\|}}.
\end{eqnarray*}
Next, we should give a upper bound for $\langle\delta^k, u^{k+1}-u^*\rangle$:
\begin{equation}\label{discrete inner product estimates}
\begin{aligned}
\langle\delta^k, u^{k+1}-u^*\rangle&\leq\|\delta^k\|(\|u^{k+1}-z^{k+1}\|+\|z^{k+1}-z^*\|)
              \leq\left(\left(1+\frac{2}{\sqrt{\tau}}\right)\frac{2\sqrt{2}\bar\rho}{\sqrt{\tau\sigma\|M_h^{-1}\|}}\right)\|\delta^k\|\equiv \bar\eta\|\delta^k\|.
\end{aligned}
\end{equation}
Then by (\ref{inequlaities property1}) in Proposition \ref{descent proposition}, we have
\begin{equation}\label{summable1}
  \begin{aligned}
  \sum\limits^{\infty}_{k=0}\left(\frac{\sigma}{2}\|r^{k+1}\|^2_{W_h-\tau M_h}
  +\frac{\sigma}{2}\|u^{k+1}-z^k\|^2_{M_h}\right)&\leq \sum\limits^{\infty}_{k=0}(\theta^{k}-\theta^{k+1})
  +\sum\limits^{\infty}_{k=0}\langle\delta^k,u^{k+1}-u^*\rangle \\
  &\leq \theta^0+\bar\eta\sum\limits^{\infty}_{k=0}\|\delta^k\|
\leq\theta^0+\bar\eta\sum\limits^{\infty}_{k=0}\epsilon^k=\theta^0+\bar\eta C_1.
  \end{aligned}
\end{equation}
Hence,
\begin{equation}\label{summable2}
  \begin{aligned}
\sum\limits^{\infty}_{k=0}\|r^{k+1}\|^2\leq \frac{2(\theta^0+\bar\eta C_1)}{\sigma\|(W_h-\tau M_h)^{-1}\|},\quad
\sum\limits^{\infty}_{k=0}\|u^{k+1}-z^k\|^2_{M_h}\leq \frac{2(\theta^0+\bar\eta C_1)}{\sigma}.
  \end{aligned}
\end{equation}
By substituting (\ref{summable2}) to (\ref{discrete KKT function for iteration}), we have
\begin{equation}\label{summable3}
  \begin{aligned}
\sum\limits^{\infty}_{k=0}R_h(w^{k+1})&\leq2\sum\limits^{\infty}_{k=0}\|\delta^k\|^2
+\eta\sum\limits^{\infty}_{k=0}\|r^{k+1}\|^2+2\sigma^2 \|M_h\|\sum\limits^{\infty}_{k=0}\|u^{k+1}-z^{k}\|_{M_h}^2\\
&\leq C:=2C_2+\eta\frac{2(\theta^0+\bar\eta C_1)}{\sigma\|(W_h-\tau M_h)^{-1}\|}+2\sigma^2 \|M_h\|\frac{2(\theta^0+\bar\eta C_1)}{\sigma}
  \end{aligned}
\end{equation}
Thus, by Lemma \ref{complexity lemma}, we know (\ref{discrete iteration complexity1}) holds. Therefore, combining the obtained global convergence results, we complete the whole proof of the Theorem \ref{discrete convergence results}.
\end{proof}
\end{theorem}

\subsection{Numerical computation of the $u$-subproblem of Algorithm \ref{algo4:inexact heterogeneous ADMM for problem RHP}}\label{subsection linear sysytem}
\subsubsection{Error analysis of the linear system {(\ref{eqn:saddle point4})}}
As we know, the linear system (\ref{eqn:saddle point4}) is a special case of the generalized saddle-point problem, thus some Krylov-based methods could be employed to inexactly solve the linear system. Let $(r^k_1, r^k_2)$ be the residual error vector, that means:
\begin{equation}\label{eqn:saddle point4 with error}
\left[
  \begin{array}{cc}
    \frac{1}{0.5\alpha+\sigma}M_h & K_h \\
    -K_h & M_h
  \end{array}
\right]\left[
         \begin{array}{c}
           y^{k+1} \\
           u^{k+1}
         \end{array}
       \right]=\left[
                 \begin{array}{c}
                   \frac{1}{0.5\alpha+\sigma}(K_h(\sigma z^k-\lambda^k)+M_hy_d)+r_1\\
                   -M_hy_c+r_2
                 \end{array}
               \right],
\end{equation}
and $\delta^k=(0.5\alpha+\sigma)M_hK_h^{-1}r_1^k+M_hK_h^{-1}M_hK_h^{-1}r_2^k$, thus in the numerical implementation we require
\begin{equation}\label{error estimates1}
  \|r^k_1\|_{2}+\|r^k_2\|_{2}\leq\frac{\epsilon_k}{\sqrt{2}\|M_hK_h^{-1}\|_2\max\{\|M_hK_h^{-1}\|_{2},0.5\alpha+\sigma\}}
\end{equation}
to guarantee the error vector $\|{\delta}^k\|_{2} \leq {\epsilon_k}$.

\subsubsection{An efficient precondition techniques for solving the linear systems}\label{precondition}
To solve (\ref{eqn:saddle point4}), in this paper, we use the generalized minimal residual {\rm (GMRES)} method. In order to speed up the convergence of the {\rm GMRES} method, the preconditioned variant of modified hermitian and skew-hermitian splitting {\rm(PMHSS)} preconditioner $\mathcal{P}$ is employed which is introduced in {\rm \cite{Bai}}:
\begin{equation}\label{precondition matrix1}
  \mathcal{P_{HSS}}=\frac{1}{\gamma}\left[
    \begin{array}{lcc}
I & \quad\sqrt{\gamma}I \\
      -\sqrt{\gamma}I & \quad\gamma I \\
    \end{array}
  \right]\left[
           \begin{array}{lcc}
             M_h+\sqrt{\gamma}K_h & 0 \\
             0 & M_h+\sqrt{\gamma}K_h \\
           \end{array}
         \right],
\end{equation}
where $\gamma=0.5\alpha+\sigma$. Let $\mathcal{A}$ denote the coefficient matrix of linear system (\ref{eqn:saddle point4}).
%

In our numerical experiments, the approximation $\widehat{G}$ corresponding to the matrix $G:=M_h+\sqrt{\gamma}K_h$ is implemented by 20 steps of Chebyshev semi-iteration when the parameter $\gamma$ is small, since in this case the coefficient matrix $G$ is dominated by the mass matrix and 20 steps of Chebyshev semi-iteration is an appropriate approximation for the action of $G$'s inverse. For more details on the Chebyshev semi-iteration method we refer to {\rm\cite{ReDoWa,chebysevsemiiteration}}. Meanwhile, for the large values of $\gamma$, the stiffness matrix $K_h$ makes a significant contribution. Hence, a fixed number of Chebyshev semi-iteration is no longer sufficient to approximate the action of $G^{-1}$. In this case, the way to avoid this difficulty is to approximate the action of $G^{-1}$ with two AMG V-cycles, which obtained by the \textbf{amg} operator in the iFEM software package\footnote{\noindent \textrm{For more details about the iFEM software package, we refer to the website \url{http://www.math.uci.edu/~chenlong/programming.html} }}.

%
\subsubsection{Terminal condition}\label{terminal condition}
Let $\epsilon$ be a given accuracy tolerance. Thus we terminate our ihADMM method when $\eta\leq\epsilon$,
where
  $\eta=\max{\{\eta_1,\eta_2,\eta_3,\eta_4,\eta_5\}}$,
in which
\begin{equation*}
  \begin{aligned}
     & \eta_1=\frac{\|K_hy- M_hu-M_h y_c\|}{1+\|M_hy_c\|},\quad \eta_2=\frac{\|M_h(u-z)\|}{1+\|u\|},\quad
     \eta_3=\frac{\|M_h(y- y_d)+ K_hp\|}{1+\|M_hy_d\|},\\
 &\eta_4=\frac{\|0.5\alpha M_hu -M_hp + M_h\lambda\|}{1+\|u\|},\quad
     \eta_5=\frac{\|z-{\rm\Pi}_{[a,b]}\left(\frac{\alpha}{2}{\rm soft}(W_h^{-1}M_h\lambda, \beta\right)\|}{1+\|u\|}.
  \end{aligned}
\end{equation*}
\subsection{A two-phase strategy for discrete problems}\label{terminal condition for PDAS}
In this section, we introduce the primal-dual active set (PDAS) method as a Phase-II algorithm to solve the discretized problem.

For problem (\ref{equ:approx discretized matrix-vector form}), eliminating artificial variable $z$, we have
\begin{equation}\label{approx discretized matrix-vector form for PDAS}
 \left\{ \begin{aligned}
        &\min \limits_{(y,u)\in\mathbb{R}^{2N_h}}^{}\frac{1}{2}\|y-y_d\|_{M_h}^{2}+\frac{\alpha}{4}\|u\|_{M_h}^{2}+\frac{\alpha}{4}\|u\|_{W_h}^{2}+\beta\|W_hu\|_1 \\
        &\quad \quad {\rm{s.t.}}\qquad K_hy=M_hu+M_hy_c  \\
          &\qquad \qquad  \quad\quad u\in [a,b]^{N_h}
                          \end{aligned} \right.\tag{$\mathrm{\overline{P}}_{h}$}
 \end{equation}
The full numerical scheme is summarized in Algorithm \ref{algo1:discrete Primal-Dual Active Set (PDAS) method}:

\begin{algorithm}[H]
  \caption{Primal-Dual Active Set (PDAS) method}
  \label{algo1:discrete Primal-Dual Active Set (PDAS) method}
Initialization: Choose $y^0$, $u^0$, $p^0$ and $\mu^0$. Set $k=0$ and $c>0$.
\begin{description}
\item[Step 1] Determine the following subsets
\begin{equation*}\begin{aligned}
    \mathcal{A}^{k+1}_a &= \{i: u_i^{k}+c(\mu_i^{k}+w_i\beta)-a<0\},\quad
  \mathcal{A}^{k+1}_b = \{i: u^{k}_i+c(\mu^{k}_i-w_i\beta)-b>0\}, \\
  \mathcal{A}^{k+1}_0 &= \{i: |u^{k}_i+c\mu^{k}_i|<cw_i\beta)\}, \quad
  \mathcal{I}^{k+1}_+ = \{i: cw_i\beta<u^{k}_i+c\mu^{k}_i<b+cw_i\beta)\}, \\
  \mathcal{I}^{k+1}_- &= \{i: a-cw_i\beta<u^{k}_i+c\mu^{k}_i<-cw_i\beta)\}.
  \end{aligned}
\end{equation*}
\item[Step 2] Solve the following system
\begin{equation*}
\left\{\begin{aligned}
        &K_hy^{k+1} - M_hu^{k+1}=0, \\
  &K_hp^{k+1} + M_h(y^{k+1}-y_d)=0,\\
  &\alpha T_hu^{k+1}-M_hp^{k+1}+\mu^{k+1} = 0,
 \end{aligned}\right.
\end{equation*}
 where $T_h=\frac{1}{2}(M_h+W_h)$, and
\begin{equation*}
 u^{k+1}=\left\{\begin{aligned}
  a\quad {\rm a.e.\ on}~ \mathcal{A}^{k+1}_{a}\\
  b\quad {\rm a.e.\ on}~ \mathcal{A}^{k+1}_{b}\\
  0\quad {\rm a.e.\ on}~ \mathcal{A}^{k+1}_{0},
   \end{aligned}\right. \qquad and\quad    \mu_i^{k+1}=\left\{\begin{aligned}
  -w_i\beta\quad {\rm a.e.\ on}~ \mathcal{I}^{k+1}_{-}\\
  w_i\beta\quad {\rm a.e.\ on}~ \mathcal{I}^{k+1}_{+}\\
    \forall i=1,2,...,N_h.
   \end{aligned}\right.
\end{equation*}
  \item[Step 3] If a termination criterion is not met, set $k:=k+1$ and go to Step 1
\end{description}
\end{algorithm}

In actual numerical implementations, let $\epsilon$ be a given accuracy tolerance. Thus we terminate our Phase-II algorithm (PDAS method) when $\eta\leq\epsilon$,
where $\eta=\max{\{\eta_1,\eta_2,\eta_3\}}$ and
\begin{equation*}
  \begin{aligned}
     & \eta_1=\frac{\|K_hy- M_hu-M_h y_c\|}{1+\|M_hy_c\|},\quad\quad\eta_2=\frac{\|M_h(y- y_d)+ K_hp\|}{1+\|M_hy_d\|}, \\
     &\eta_3=\frac{\|u-{\rm\Pi}_{[a,b]}\left(\frac{\alpha}{2}{\rm soft}(W_h^{-1}M_h(p-u), \beta\right)\|}{1+\|u\|}.
  \end{aligned}
\end{equation*}

\subsection{\textbf {Algorithms for comparison}}\label{subsection5.4}
In this section, in order to show the high efficiency of our ihADMM and two-phase strategy, we introduce the details of some mentioned existing methods for sparse optimal control problems.

As a comparison, one can only employ the PDAS method to solve (\ref{approx discretized matrix-vector form for PDAS}). An important issue for the successful application of the PDAS scheme, is the use of a robust line-search method for globalization purposes. However, since there exist a nonsmooth term $\beta\|W_hu\|_1$ in the objective function of (\ref{approx discretized matrix-vector form for PDAS}), we do not have differentiability (in the classical sense) of the minimizing function and the classical Armijo, Wolfe and Goldstein line search schemes can not be used. To overcome this difficulty, an alternative approach, i.e., the derivative-free line-search (DFLS) procedure, is used. For more details of DFLS, one can refer to \cite{DFLS1}. Then a globalized version of PDAS with DFLS is given.
In addition, as we have mentioned in Section \ref{intro}, instead of our ihADMM method and PDAS method, one can also apply the APG method \cite{FIP} to solve problem (\ref{approx discretized matrix-vector form for PDAS}) for the sake of numerical comparison, see \cite{FIP} for more details of the APG method.

\section{Numerical results}\label{sec:5}
In this section, we will use the following example to evaluate the numerical behaviour of our ihADMM and two-phase strategy for problem (\ref{equ:approx discretized matrix-vector form}) and verify the theoretical error estimates given in Section \ref{sec:3}. For comparison, we will also show the numerical results obtained by the classical ADMM and the APG algorithm, and the PDAS with line search.

\subsection{Algorithmic Details}\label{subsec:5.1}
\quad

\textbf{Discretization.} As show in Section \ref{sec:3}, the discretization was carried out by using piecewise linear and continuous finite elements. 
The assembly of mass and the stiffness matrices, as well as the lump mass matrix was left to the iFEM software package. To present the finite element error estimates results, it is convenient to introduce the experimental order of convergence (EOC), which for some positive error functional $E(h)$ with $h> 0$ is defined as follows: Given two grid sizes $h_1\neq h_2$, let
\begin{equation}\label{EOC}
  \mathrm{EOC}:=\frac{\log {E(h_1)}-\log {E(h_2)}}{\log{h_1}-\log {h_2}}.
\end{equation}
It follows from this definition that if $E(h)=\mathcal{O}(h^{\gamma})$ then $\mathrm{EOC}\approx\gamma$. The error functional $E(\cdot)$ investigated in the present section is given by
  $E_2(h):=\|u-u_h\|_{L^2{(\Omega)}}$.

\textbf{Initialization.} For all numerical examples, we choose $u=0$ as initialization $u^0$ for all algorithms.

\textbf{Parameter Setting.} For the classical ADMM and our ihADMM, the penalty parameter $\sigma$ was chosen as $\sigma=0.1 \alpha$. About the step-length $\tau$, we choose $\tau=1.618$ for the classical ADMM, and $\tau=1$ for our ihADMM. For the PDAS method, the parameter in the active set strategy was chosen as $c=1$. For the APG method, we estimate an approximation for the Lipschitz constant $L$ with a backtracking method.

\textbf{Terminal Condition.} In our numerical experiments, we measure the accuracy of an approximate optimal solution by using the corresponding K-K-T residual error for each algorithm. For the purpose of showing the efficiency of our ihADMM, we report the numerical results obtained by running the classical ADMM and the APG method to compare with the results obtained by our ihADMM. In this case, we terminate all the algorithms when $\eta<10^{-6}$ with the maximum number of iterations set at 500. Additionally, we also employ our two-phase strategy to obtain more accurate solution. As a comparison, a globalized version of the PDAS algorithm are also shown. In this case, we terminate the our ihADMM when $\eta<10^{-3}$ to warm-start the PDAS algorithm which is terminated when $\eta<10^{-10}$. Similarly, we terminate the PDAS algorithm with DFLS when $\eta<10^{-10}$.

\textbf{Computational environment.}
All our computational results are obtained by MATLAB Version 8.5(R2015a) running on a computer with
64-bit Windows 7.0 operation system, Intel(R) Core(TM) i7-5500U CPU (2.40GHz) and 8GB of memory.
\subsection{Examples}\label{subsec:5.2}
%

\begin{example}\label{example:1}
  \begin{equation*}
     \left\{ \begin{aligned}
        &\min \limits_{(y,u)\in H^1_0(\Omega)\times L^2(\Omega)}^{}\ \ J(y,u)=\frac{1}{2}\|y-y_d\|_{L^2(\Omega)}^{2}+\frac{\alpha}{2}\|u\|_{L^2(\Omega)}^{2}+\beta\|u\|_{L^1(\Omega)} \\
        &\qquad\quad{\rm s.t.}\qquad \quad\quad\quad-\Delta y=u+y_c\quad \mathrm{in}\  \Omega,\\
         &\qquad \qquad \quad\qquad \qquad \qquad  ~y=0\quad  \mathrm{on}\ \partial\Omega,\\
         &\quad\qquad \qquad \qquad\qquad  u\in U_{ad}=\{v(x)|a\leq v(x)\leq b, {\rm a.e }\  \mathrm{on}\ \Omega \}.
                          \end{aligned} \right.
 \end{equation*}\end{example}
Here, we consider the problem with control $u\in L^2(\Omega)$ on the unit square $\Omega= (0, 1)^2$ with $\alpha=0.5, \beta=0.5$, $a=-0.5$ and $b=0.5$. It is a constructed problem, thus we set $y^*=\sin(\pi x_1)\sin(\pi x_2)$ and $p^*=2\beta\sin(2\pi x_1)\exp(0.5x_1)\sin(4\pi x_2)$. Then through $u^*=\mathrm{\Pi}_{U_{ad}}\left(\frac{1}{\alpha}{\rm{soft}}\left(p^*,\beta\right)\right)$, $y_c=y^*-\mathcal{S}u^*$ and $y_d=\mathcal{S}^{-*}p^{*}+y^*$, we can construct the example for which we know the exact solution.


An example for the discretized optimal control on mesh $h=2^{-7}$ is shown in Figure \ref{example1fig:control on h=$2^{-7}$}. The error of the control $u$ w.r.t the $L^2$-norm and the experimental order of convergence (EOC) for control are presented in Table \ref{tab:1}. They also confirm that indeed the convergence rate is of order $O(h)$.

Numerical results for the accuracy of solution, number of iterations and cpu time obtained by our ihADMM, classical ADMM and APG methods are shown in Table \ref{tab:1}. As a result from Table \ref{tab:1}, we can see that our proposed ihADMM method is an efficient algorithm to solve problem (\ref{equ:approx discretized matrix-vector form}) to medium accuracy. Moreover, it is obvious that our ihADMM outperform the classical ADMM and the APG method in terms of in CPU time, especially when the discretization is in a fine level. It is worth noting that although the APG method require less number of iterations when the termination condition is satisfied, the APG method spend much time on backtracking step with the aim of finding an appropriate approximation for the Lipschitz constant. This is the
reason that our ihADMM has better performance than the APG method in actual numerical implementation. Furthermore, the numerical results in terms of iteration numbers illustrate the mesh-independent performance of the ihADMM and the APG method, except for the classical ADMM.

In addition, to obtain more accurate solution, we employ our two-phase strategy. The numerical results are shown in Table \ref{tab:2}. In order to show our the power and the importance of our two-phase framework, as a comparison, numerical results obtained by the PDAS with line search are also shown in Table \ref{tab:2}. It can be observed that our two-phase strategy is faster and more efficient than the PDAS with line search in terms of the iteration numbers and CPU time.

\begin{figure}[H]\scriptsize
\begin{center}
\includegraphics[width=0.40\textwidth]{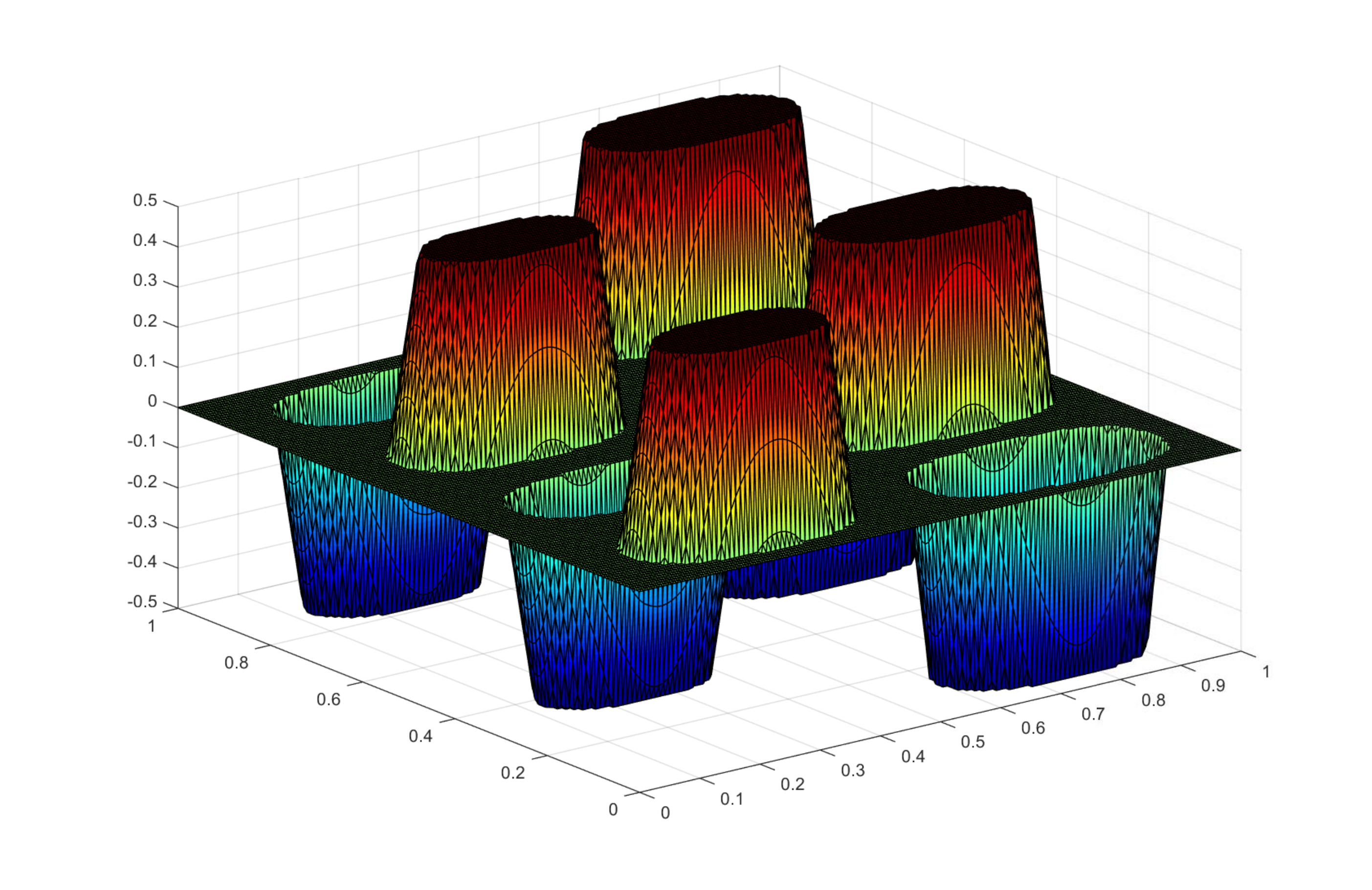}
\includegraphics[width=0.40\textwidth]{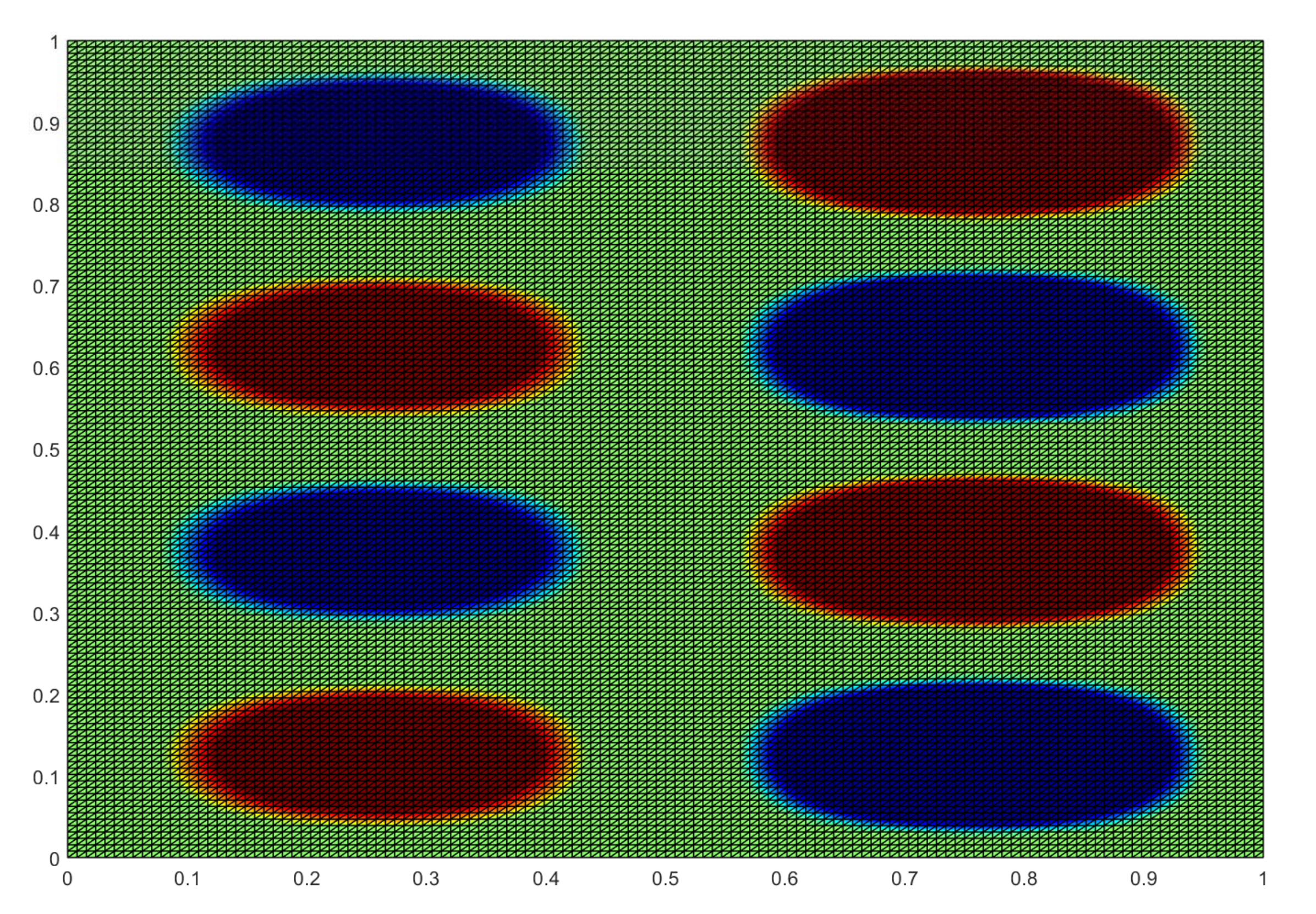}
\caption{Optimal control $u_h$ on the square, $h=2^{-7}$. Dark red and dark blue areas correspond to $u_h=\pm 0.5$ and green areas to $u_h=0$}\label{example1fig:control on h=$2^{-7}$}.
\end{center}
\end{figure}

\begin{example}\label{example:2}\rm{\cite[Example 1]{Stadler}}
  \begin{equation*}
     \left\{ \begin{aligned}
        &\min \limits_{(y,u)\in Y\times U}^{}\ \ J(y,u)=\frac{1}{2}\|y-y_d\|_{L^2(\Omega)}^{2}+\frac{\alpha}{2}\|u\|_{L^2(\Omega)}^{2}+{\beta}\|u\|_{L^1(\Omega)} \\
        &\quad{\rm s.t.}\qquad -\Delta y=u,\quad \mathrm{in}\  \Omega=(0,1)\times(0,1) \\
         &\qquad \qquad \qquad y=0,\quad  \mathrm{on}\ \partial\Omega\\
         &\qquad \qquad\qquad  u\in U_{ad}=\{v(x)|a\leq v(x)\leq b, {\rm a.e }\  \mathrm{on}\ \Omega \},
                          \end{aligned} \right.
 \end{equation*}
\end{example}
where the desired state $y_d=\frac{1}{6}\sin(2\pi x)\exp(2x)\sin(2\pi y)$, the parameters $\alpha=10^{-5}$, $\beta=10^{-3}$, $a=-30$ and $b=30$. In addition, the exact solutions of the problem is unknown. Instead we use the numerical solutions computed on a grid with $h^*=2^{-10}$ as reference solutions.

An example for the discretized optimal control on mesh $h=2^{-7}$ is displayed in Figure \ref{fig:discretized control on h=$2^{-7}$}. The error of the control $u$ w.r.t the $L^2$ norm with respect to the solution on the finest grid ($h^*=2^{-10}$) and the experimental order of convergence (EOC) for control are presented in Table \ref{tab:3}. They confirms the linear rate of convergence w.r.t. $h$ as proved in Theorem \ref{theorem:error2} and Corollary \ref{corollary:error1}.

Numerical results for the accuracy of solution, number of iterations and cpu time obtained by our ihADMM, classical ADMM and APG methods are also shown in Table \ref{tab:3}. Experiment results show that the ADMM has evident advantage over the classical ADMM and the APG method in computing time. Furthermore, the numerical results in terms of iteration numbers also illustrate the mesh-independent performance of our ihADMM. In addition, in Table \ref{tab:4}, we give the numerical results obtained by our two-phase strategy and the PDAS method with line search. As a result from Table \ref{tab:4}, it can be observed that our two-phase strategy outperform the PDAS with line search in terms of the CPU time. These results demonstrate that our ihADMM is highly efficient in obtaining  an approximate solution with moderate accuracy. And our two-phase strategy could represent an effective alternative to PDAS method.

\begin{figure}[H]
\centering
\includegraphics[width=0.35\textwidth]{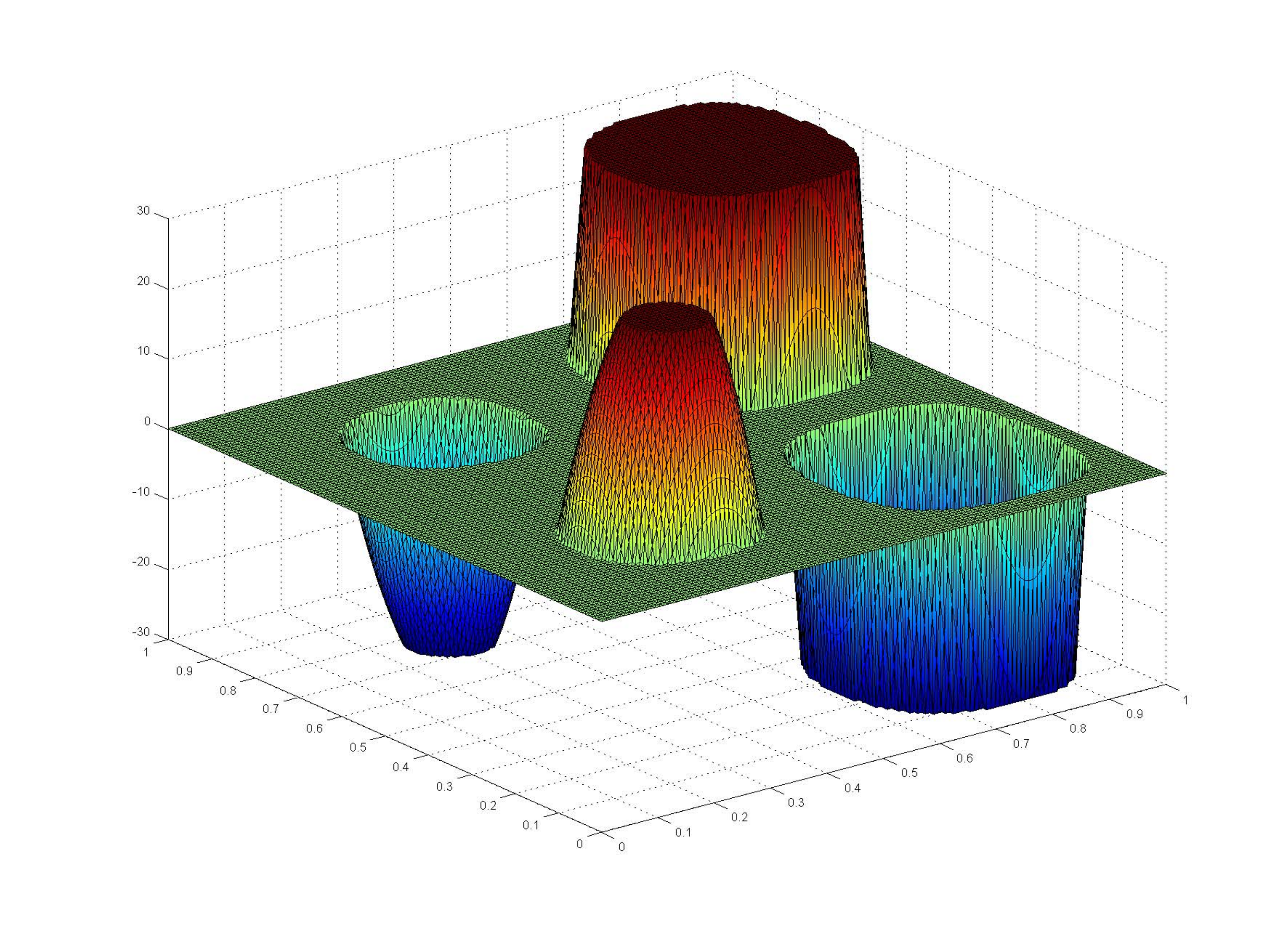}
\includegraphics[width=0.35\textwidth]{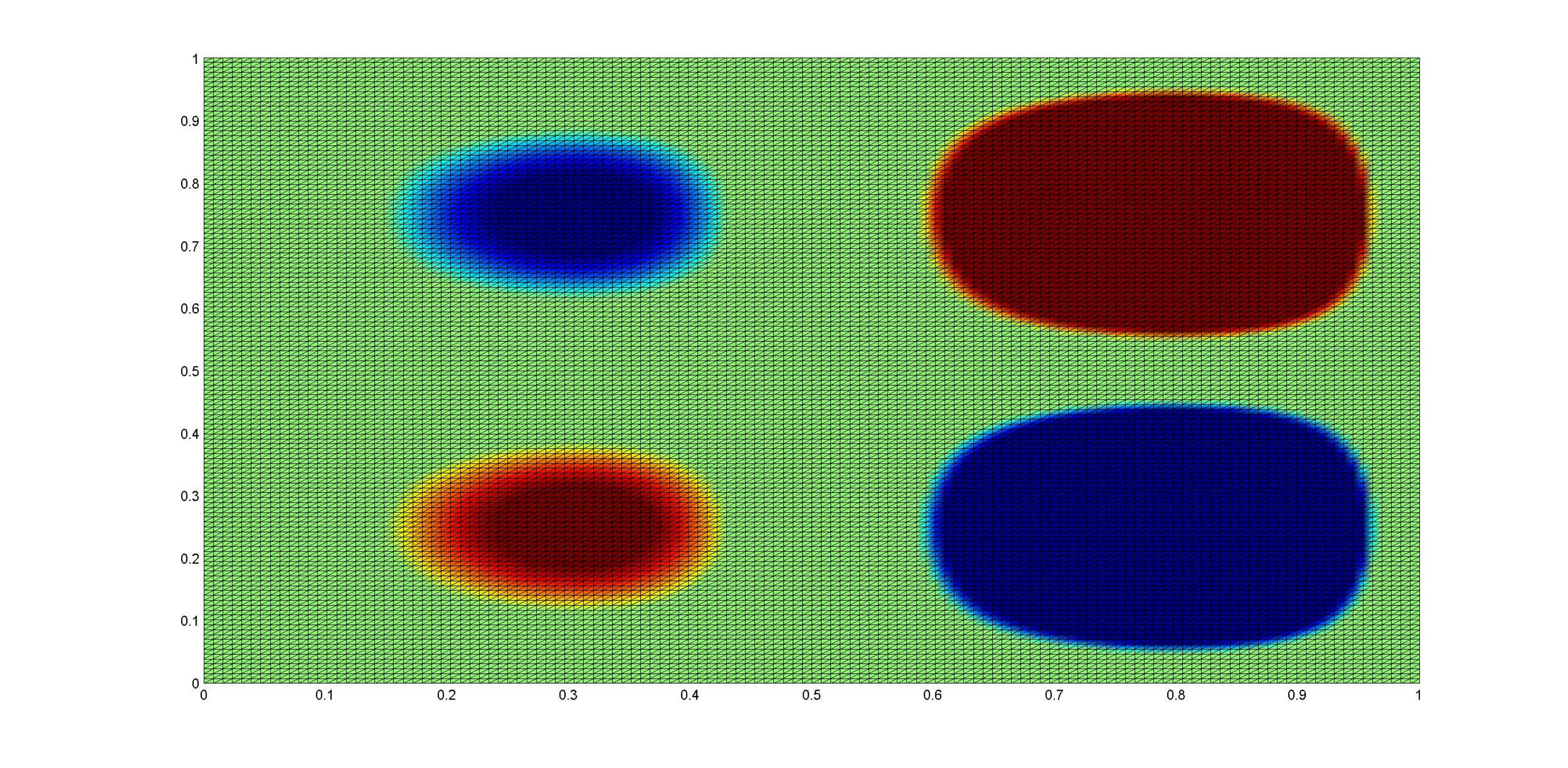}
\caption{Optimal control $u_h$ on the square, $h=2^{-7}$. Dark red and dark blue areas correspond to $u_h=\pm 30$ and green areas to $u_h=0$}\label{fig:discretized control on h=$2^{-7}$}
\end{figure}

\section{Concluding remarks}\label{sec:6}
In this paper, elliptic PDE-constrained optimal control problems with $L^1$-control cost ($L^1$-EOCP) are considered. In order to make discretized problems have a decoupled form, instead of directly using the standard piecewise linear finite element to discretize the problem, we utilize nodal quadrature formulas to approximately discretize the $L^1$-norm and $L^2$-norm. It was proven that these approximation steps do not change the order of error estimates. By taking advantage of inherent structures of the problem, we proposed an inexact heterogeneous ADMM (ihADMM) to solve discretized problems. Furthermore, theoretical results on the global convergence as well as the iteration complexity results $o(1/k)$ for ihADMM were given. Moreover, in order to obtain more accurate solution, a two-phase strategy was introduced, in which the primal-dual active set (PDAS) method is used as a postprocessor of the ihADMM. Numerical results demonstrated the efficiency of our ihADMM and the two-phase strategy.
\section*{Acknowledgments}
The authors would like to thank Dr. Long Chen for the FEM package
iFEM \cite{Chen} in Matlab and also would like to thank the colleagues for their valuable suggestions that led to improvement in this paper.


\begin{table}[H]
\caption{Example \ref{example:1}: The convergence behavior of our ihADMM, classical ADMM and APG for (\ref{equ:approx discretized matrix-vector form}). In the table, $\#$dofs stands for the number of degrees of freedom for the control variable on each grid level.}\label{tab:1}
\begin{center}
\begin{tabular}{@{\extracolsep{\fill}}|c|c|c|c|c|c|c|c|}
\hline
\hline
\multirow{2}{*}{$h$}&\multirow{2}{*}{$\#$dofs}&\multirow{2}{*}{$E_2$}&\multirow{2}{*}{EOC}& \multirow{2}{*}{Index} &\multirow{2}{*}{ihADMM} & \multirow{2}{*}{classical ADMM} & \multirow{2}{*}{APG} \\
&&&&&&&\\
\hline

                            &&&&\multirow{2}{*}{iter}            &\multirow{2}{*}{27} &\multirow{2}{*}{32} &\multirow{2}{*}{13}              \\
                            &&&&&&&\\
\multirow{2}{*}{$2^{-3}$}   &\multirow{2}{*}{49}&\multirow{2}{*}{0.3075}&\multirow{2}{*}{--}
                            &\multirow{2}{*}{residual $\eta$}
                            &\multirow{2}{*}{7.15e-07}            &\multirow{2}{*}{7.55e-07}
                            &\multirow{2}{*}{6.88e-07}             \\
                            &&&&&&&\\
                            &&&&\multirow{2}{*}{CPU time/s}          &\multirow{2}{*}{0.19} &\multirow{2}{*}{0.23} &\multirow{2}{*}{0.18}             \\
                            &&&&&&&\\
\hline
                            &&&&\multirow{2}{*}{iter}            &\multirow{2}{*}{31} &\multirow{2}{*}{44} &\multirow{2}{*}{13}             \\
                            &&&&&&&\\
\multirow{2}{*}{$2^{-4}$}   &\multirow{2}{*}{225}&\multirow{2}{*}{0.1237}&\multirow{2}{*}{1.3137}&\multirow{2}{*}{residual $\eta$}
                            &\multirow{2}{*}{9.77e-07}            &\multirow{2}{*}{9.91e-07}
                            &\multirow{2}{*}{8.23e-07}             \\
                            &&&&&&&\\
                            &&&&\multirow{2}{*}{CPU times/s}             &\multirow{2}{*}{0.37} &\multirow{2}{*}{0.66} &\multirow{2}{*}{0.32}  \\
                            &&&&&&&\\
\hline
                            &&&&\multirow{2}{*}{iter}            &\multirow{2}{*}{31} &\multirow{2}{*}{58} &\multirow{2}{*}{12}            \\
                            &&&&&&&\\
\multirow{2}{*}{$2^{-5}$}   &\multirow{2}{*}{961}&\multirow{2}{*}{0.0516}&\multirow{2}{*}{1.2870}
                            &\multirow{2}{*}{residual $\eta$}
                            &\multirow{2}{*}{7.41e-07}&\multirow{2}{*}{8.11e-07}
                            &\multirow{2}{*}{7.58e-07}             \\
                            &&&&&&&\\
                            &&&&\multirow{2}{*}{CPU time/s}          &\multirow{2}{*}{1.02} &\multirow{2}{*}{2.32} &\multirow{2}{*}{1.00}            \\
                            &&&&&&&\\
\hline
                            &&&&\multirow{2}{*}{iter}            &\multirow{2}{*}{32} &\multirow{2}{*}{76} &\multirow{2}{*}{14}              \\
                            &&&&&&&\\
\multirow{2}{*}{$2^{-6}$}   &\multirow{2}{*}{3969}&\multirow{2}{*}{0.0201}&\multirow{2}{*}{1.3112}&\multirow{2}{*}{residual $\eta$}
                            &\multirow{2}{*}{7.26e-07}&\multirow{2}{*}{8.10e-07}
                            &\multirow{2}{*}{7.88e-07}             \\
                            &&&&&&&\\
                            &&&&\multirow{2}{*}{CPU time/s}          &\multirow{2}{*}{4.18} &\multirow{2}{*}{9.12} &\multirow{2}{*}{4.25}            \\
                            &&&&&&&\\
\hline
                            &&&&\multirow{2}{*}{iter}            &\multirow{2}{*}{31} &\multirow{2}{*}{94} &\multirow{2}{*}{14}              \\
                            &&&&&&&\\
\multirow{2}{*}{$2^{-7}$}   &\multirow{2}{*}{16129}&\multirow{2}{*}{0.0078}&\multirow{2}{*}{1.3252}&\multirow{2}{*}{residual $\eta$}
                            &\multirow{2}{*}{ 5.33e-07}&\multirow{2}{*}{7.85e-07}
                            &\multirow{2}{*}{4.45e-07}             \\
                            &&&&&&&\\
                            &&&&\multirow{2}{*}{CPU time/s}          &\multirow{2}{*}{17.72} &\multirow{2}{*}{65.82} &\multirow{2}{*}{26.25}            \\
                            &&&&&&&\\
\hline
                            &&&&\multirow{2}{*}{iter}            &\multirow{2}{*}{32} &\multirow{2}{*}{127} &\multirow{2}{*}{13}              \\
                            &&&&&&&\\
\multirow{2}{*}{$2^{-8}$}   &\multirow{2}{*}{65025}&\multirow{2}{*}{0.0026}&\multirow{2}{*}{1.3772}
                            &\multirow{2}{*}{residual $\eta$}
                            &\multirow{2}{*}{6.88e-07}&\multirow{2}{*}{8.93e-07}
                            &\multirow{2}{*}{7.47e-07}             \\
                            &&&&&&&\\
                            &&&&\multirow{2}{*}{CPU time/s}          &\multirow{2}{*}{70.45} &\multirow{2}{*}{312.65} &\multirow{2}{*}{80.81}            \\
                            &&&&&&&\\
\hline
                            &&&&\multirow{2}{*}{iter}            &\multirow{2}{*}{31} &\multirow{2}{*}{255} &\multirow{2}{*}{13}              \\
                            &&&&&&&\\
\multirow{2}{*}{$2^{-9}$}   &\multirow{2}{*}{261121}&\multirow{2}{*}{0.0009}&\multirow{2}{*}{1.4027}
                            &\multirow{2}{*}{residual $\eta$}
                            &\multirow{2}{*}{7.43e-07}&\multirow{2}{*}{7.96e-07}
                            &\multirow{2}{*}{6.33e-07}             \\
                            &&&&&&&\\
                            &&&&\multirow{2}{*}{CPU time/s}          &\multirow{2}{*}{525.28} &\multirow{2}{*}{4845.31} &\multirow{2}{*}{620.55}            \\
                            &&&&&&&\\
\hline
\end{tabular}
\end{center}
\end{table}

\begin{table}[H]
\caption{Example \ref{example:1}: The convergence behavior of our two-phase strategy, PDAS with line search.}\label{tab:2}
\begin{center}
\begin{tabular}{@{\extracolsep{\fill}}|c|c|c|cc|c|c|c|}
\hline
\hline
\multirow{2}{*}{$h$}&\multirow{2}{*}{$\#$dofs}& \multirow{2}{*}{Index of performance} &\multicolumn{2}{c|}{Two-Phase strategy}& \multirow{2}{*}{PDAS with line search} \\
\cline{4-5}
& & &ihADMM\ $+$\ PDAS&& \\
\hline
                            &&\multirow{2}{*}{iter}            &\multirow{2}{*}{13\quad $+$\quad 5} &&\multirow{2}{*}{21} \\
                            &&&&&\\
\multirow{2}{*}{$2^{-3}$}   &\multirow{2}{*}{49}
                            &\multirow{2}{*}{residual $\eta$}
                            &\multirow{2}{*}{8.55e-12}            & &\multirow{2}{*}{7.88e-12} \\
                            &&&&&\\
                            &&\multirow{2}{*}{CPU time/s}          &\multirow{2}{*}{0.17} &&\multirow{2}{*}{0.32} \\
                            &&&&&\\
\hline
                            &&\multirow{2}{*}{iter}            &\multirow{2}{*}{13\quad $+$\quad 6} &&\multirow{2}{*}{22} \\
                            &&&&&\\
\multirow{2}{*}{$2^{-4}$}   &\multirow{2}{*}{225}&\multirow{2}{*}{residual $\eta$}
                            &\multirow{2}{*}{1.24e-11}           & &\multirow{2}{*}{1.87e-11}
            \\
                            &&&&&\\
                            &&\multirow{2}{*}{CPU times/s}             &\multirow{2}{*}{0.27} &&\multirow{2}{*}{0.54}  \\
                            &&&&&\\
\hline
                            &&\multirow{2}{*}{iter}            &\multirow{2}{*}{14\quad $+$\quad 5} &&\multirow{2}{*}{22} \\
                            &&&&&\\
\multirow{2}{*}{$2^{-5}$}   &\multirow{2}{*}{961}
                            &\multirow{2}{*}{residual $\eta$}
                            &\multirow{2}{*}{8.10e-12}          & &\multirow{2}{*}{8.42e-12}            \\
                            &&&&&\\
                            &&\multirow{2}{*}{CPU time/s}          &\multirow{2}{*}{0.95} &&\multirow{2}{*}{2.07} \\
                            &&&&&\\
\hline
                            &&\multirow{2}{*}{iter}            &\multirow{2}{*}{14\quad $+$\quad 6} &&\multirow{2}{*}{23} \\
                            &&&&&\\
\multirow{2}{*}{$2^{-6}$}   &\multirow{2}{*}{3969}&\multirow{2}{*}{residual $\eta$}
                            &\multirow{2}{*}{4.15e-12}          &  &\multirow{2}{*}{4.00e-12}             \\
                            &&&&&\\
                            &&\multirow{2}{*}{CPU time/s}          &\multirow{2}{*}{3.65} &&\multirow{2}{*}{6.98}\\
                            &&&&&\\
\hline
                            &&\multirow{2}{*}{iter}            &\multirow{2}{*}{15\quad $+$\quad 6} &&\multirow{2}{*}{23} \\
                            &&&&&\\
\multirow{2}{*}{$2^{-7}$}   &\multirow{2}{*}{16129}&\multirow{2}{*}{residual $\eta$}
                            &\multirow{2}{*}{ 1.43e-12}        &   &\multirow{2}{*}{1.52e-12}          \\
                            &&&&&\\
                            &&\multirow{2}{*}{CPU time/s}          &\multirow{2}{*}{22.10} &&\multirow{2}{*}{43.13} \\
                            &&&&&\\
\hline
                            &&\multirow{2}{*}{iter}            &\multirow{2}{*}{15\quad $+$\quad 5} &&\multirow{2}{*}{24} \\
                            &&&&&\\
\multirow{2}{*}{$2^{-8}$}   &\multirow{2}{*}{65025}
                            &\multirow{2}{*}{residual $\eta$}
                            &\multirow{2}{*}{5.21e-12}        &     &\multirow{2}{*}{5.03e-12}             \\
                            &&&&&\\
                            &&\multirow{2}{*}{CPU time/s}          &\multirow{2}{*}{68.22} &&\multirow{2}{*}{140.18}\\
                            &&&&&\\
\hline
                            &&\multirow{2}{*}{iter}            &\multirow{2}{*}{15\quad $+$\quad 6} &&\multirow{2}{*}{24} \\
                            &&&&&\\
\multirow{2}{*}{$2^{-9}$}   &\multirow{2}{*}{261121}
                            &\multirow{2}{*}{residual $\eta$}
                            &\multirow{2}{*}{3.77e-12}       &     &\multirow{2}{*}{3.76e-12}            \\
                            &&&&&\\
                            &&\multirow{2}{*}{CPU time/s}&\multirow{2}{*}{540.57} &&\multirow{2}{*}{1145.63}\\
                            &&&&&\\
\hline
\end{tabular}
\end{center}
\end{table}

\begin{table}[H]
\caption{Example \ref{example:2}: The convergence behavior of ihADMM, classical ADMM and APG for (\ref{equ:approx discretized matrix-vector form}). 
}\label{tab:3}
\begin{center}
\begin{tabular}{@{\extracolsep{\fill}}|c|c|c|c|c|c|c|c|}
\hline
\hline
\multirow{2}{*}{$h$}&\multirow{2}{*}{$\#$dofs}&\multirow{2}{*}{$E_2$}&\multirow{2}{*}{EOC}& \multirow{2}{*}{Index} &\multirow{2}{*}{ihADMM} & \multirow{2}{*}{classical ADMM} & \multirow{2}{*}{APG} \\
&&&&&&&\\
\hline

                            &&&&\multirow{2}{*}{iter}            &\multirow{2}{*}{40} &\multirow{2}{*}{48} &\multirow{2}{*}{18}              \\
                            &&&&&&&\\
\multirow{2}{*}{$2^{-3}$}   &\multirow{2}{*}{49}&\multirow{2}{*}{0.3075}&\multirow{2}{*}{--}
                            &\multirow{2}{*}{residual $\eta$}
                            &\multirow{2}{*}{8.22e-07}            &\multirow{2}{*}{8.65e-07}
                            &\multirow{2}{*}{7.96e-07}             \\
                            &&&&&&&\\
                            &&&&\multirow{2}{*}{CPU time/s}          &\multirow{2}{*}{0.30} &\multirow{2}{*}{0.51} &\multirow{2}{*}{0.24}             \\
                            &&&&&&&\\
\hline
                            &&&&\multirow{2}{*}{iter}            &\multirow{2}{*}{41} &\multirow{2}{*}{56} &\multirow{2}{*}{18}             \\
                            &&&&&&&\\
\multirow{2}{*}{$2^{-4}$}   &\multirow{2}{*}{225}&\multirow{2}{*}{0.1237}&\multirow{2}{*}{1.3137}&\multirow{2}{*}{residual $\eta$}
                            &\multirow{2}{*}{7.22e-07}            &\multirow{2}{*}{8.01e-07}
                            &\multirow{2}{*}{7.58e-07}             \\
                            &&&&&&&\\
                            &&&&\multirow{2}{*}{CPU times/s}             &\multirow{2}{*}{0.45} &\multirow{2}{*}{0.71} &\multirow{2}{*}{0.44}  \\
                            &&&&&&&\\
\hline
                            &&&&\multirow{2}{*}{iter}            &\multirow{2}{*}{40} &\multirow{2}{*}{69} &\multirow{2}{*}{19}            \\
                            &&&&&&&\\
\multirow{2}{*}{$2^{-5}$}   &\multirow{2}{*}{961}&\multirow{2}{*}{0.0516}&\multirow{2}{*}{1.2870}
                            &\multirow{2}{*}{residual $\eta$}
                            &\multirow{2}{*}{8.12e-07}&\multirow{2}{*}{8.01e-07}
                            &\multirow{2}{*}{7.90e-07}             \\
                            &&&&&&&\\
                            &&&&\multirow{2}{*}{CPU time/s}          &\multirow{2}{*}{1.60} &\multirow{2}{*}{3.05} &\multirow{2}{*}{1.58}            \\
                            &&&&&&&\\
\hline
                            &&&&\multirow{2}{*}{iter}            &\multirow{2}{*}{42} &\multirow{2}{*}{85} &\multirow{2}{*}{18}              \\
                            &&&&&&&\\
\multirow{2}{*}{$2^{-6}$}   &\multirow{2}{*}{3969}&\multirow{2}{*}{0.0201}&\multirow{2}{*}{1.3112}&\multirow{2}{*}{residual $\eta$}
                            &\multirow{2}{*}{6.11e-07}&\multirow{2}{*}{7.80e-07}
                            &\multirow{2}{*}{6.45e-07}             \\
                            &&&&&&&\\
                            &&&&\multirow{2}{*}{CPU time/s}          &\multirow{2}{*}{7.25} &\multirow{2}{*}{14.62} &\multirow{2}{*}{7.45}            \\
                            &&&&&&&\\
\hline
                            &&&&\multirow{2}{*}{iter}            &\multirow{2}{*}{40} &\multirow{2}{*}{108} &\multirow{2}{*}{18}              \\
                            &&&&&&&\\
\multirow{2}{*}{$2^{-7}$}   &\multirow{2}{*}{16129}&\multirow{2}{*}{0.0078}&\multirow{2}{*}{1.3252}&\multirow{2}{*}{residual $\eta$}
                            &\multirow{2}{*}{ 6.35e-07}&\multirow{2}{*}{7.11e-07}
                            &\multirow{2}{*}{5.62e-07}             \\
                            &&&&&&&\\
                            &&&&\multirow{2}{*}{CPU time/s}          &\multirow{2}{*}{33.85} &\multirow{2}{*}{101.36} &\multirow{2}{*}{34.39}            \\
                            &&&&&&&\\
\hline
                            &&&&\multirow{2}{*}{iter}            &\multirow{2}{*}{41} &\multirow{2}{*}{132} &\multirow{2}{*}{19}              \\
                            &&&&&&&\\
\multirow{2}{*}{$2^{-8}$}   &\multirow{2}{*}{65025}&\multirow{2}{*}{0.0026}&\multirow{2}{*}{1.3772}
                            &\multirow{2}{*}{residual $\eta$}
                            &\multirow{2}{*}{7.55e-07}&\multirow{2}{*}{7.83e-07}
                            &\multirow{2}{*}{7.57e-07}             \\
                            &&&&&&&\\
                            &&&&\multirow{2}{*}{CPU time/s}          &\multirow{2}{*}{158.62} &\multirow{2}{*}{508.65} &\multirow{2}{*}{165.75}            \\
                            &&&&&&&\\
\hline
                            &&&&\multirow{2}{*}{iter}            &\multirow{2}{*}{42} &\multirow{2}{*}{278} &\multirow{2}{*}{18}              \\
                            &&&&&&&\\
\multirow{2}{*}{$2^{-9}$}   &\multirow{2}{*}{261121}&\multirow{2}{*}{0.0009}&\multirow{2}{*}{1.4027}
                            &\multirow{2}{*}{residual $\eta$}
                            &\multirow{2}{*}{5.25e-07}&\multirow{2}{*}{5.56e-07}
                            &\multirow{2}{*}{4.85e-07}             \\
                            &&&&&&&\\
                            &&&&\multirow{2}{*}{CPU time/s}          &\multirow{2}{*}{1781.98} &\multirow{2}{*}{11788.52} &\multirow{2}{*}{1860.11}            \\
                            &&&&&&&\\
\hline
                            &&&&\multirow{2}{*}{iter}            &\multirow{2}{*}{41} &\multirow{2}{*}{500} &\multirow{2}{*}{19}              \\
                            &&&&&&&\\
\multirow{2}{*}{$2^{-10}$}   &\multirow{2}{*}{1046529}&\multirow{2}{*}{--}&\multirow{2}{*}{1.4027}
                            &\multirow{2}{*}{residual $\eta$}
                            &\multirow{2}{*}{8.78e-07}&\multirow{2}{*}{Error}
                            &\multirow{2}{*}{8.47e-07}             \\
                            &&&&&&&\\
                            &&&&\multirow{2}{*}{CPU time/s}          &\multirow{2}{*}{42033.79} &\multirow{2}{*}{Error} &\multirow{2}{*}{44131.27}            \\
                            &&&&&&&\\
\hline
\end{tabular}
\end{center}
\end{table}

\begin{table}[H]
\caption{Example \ref{example:2}: The behavior of two-phase strategy and the PDAS method.}\label{tab:4}
\begin{center}
\begin{tabular}{@{\extracolsep{\fill}}|c|c|c|cc|c|c|c|}
\hline
\hline
\multirow{2}{*}{$h$}&\multirow{2}{*}{$\#$dofs}& \multirow{2}{*}{Index of performance} &\multicolumn{2}{c|}{Two-Phase strategy}& \multirow{2}{*}{PDAS with line search} \\
\cline{4-5}
& & &ihADMM\ $+$\ PDAS&& \\
\hline
                            &&\multirow{2}{*}{iter}            &\multirow{2}{*}{18\quad $+$\quad 8} &&\multirow{2}{*}{24} \\
                            &&&&&\\
\multirow{2}{*}{$2^{-3}$}   &\multirow{2}{*}{49}
                            &\multirow{2}{*}{residual $\eta$}
                            &\multirow{2}{*}{4.45e-12}            & &\multirow{2}{*}{4.36e-12} \\
                            &&&&&\\
                            &&\multirow{2}{*}{CPU time/s}          &\multirow{2}{*}{0.35} &&\multirow{2}{*}{0.53} \\
                            &&&&&\\
\hline
                            &&\multirow{2}{*}{iter}            &\multirow{2}{*}{18\quad $+$\quad 8} &&\multirow{2}{*}{25} \\
                            &&&&&\\
\multirow{2}{*}{$2^{-4}$}   &\multirow{2}{*}{225}&\multirow{2}{*}{residual $\eta$}
                            &\multirow{2}{*}{5.84e-12}           & &\multirow{2}{*}{6.01e-11}
            \\
                            &&&&&\\
                            &&\multirow{2}{*}{CPU times/s}             &\multirow{2}{*}{0.68} &&\multirow{2}{*}{1.02}  \\
                            &&&&&\\
\hline
                            &&\multirow{2}{*}{iter}            &\multirow{2}{*}{19\quad $+$\quad 7} &&\multirow{2}{*}{24} \\
                            &&&&&\\
\multirow{2}{*}{$2^{-5}$}   &\multirow{2}{*}{961}
                            &\multirow{2}{*}{residual $\eta$}
                            &\multirow{2}{*}{6.89e-12}          & &\multirow{2}{*}{6.87e-12}            \\
                            &&&&&\\
                            &&\multirow{2}{*}{CPU time/s}          &\multirow{2}{*}{1.98} &&\multirow{2}{*}{2.99} \\
                            &&&&&\\
\hline
                            &&\multirow{2}{*}{iter}            &\multirow{2}{*}{18\quad $+$\quad 8} &&\multirow{2}{*}{26} \\
                            &&&&&\\
\multirow{2}{*}{$2^{-6}$}   &\multirow{2}{*}{3969}&\multirow{2}{*}{residual $\eta$}
                            &\multirow{2}{*}{2.15e-11}          &  &\multirow{2}{*}{2.28e-11}             \\
                            &&&&&\\
                            &&\multirow{2}{*}{CPU time/s}          &\multirow{2}{*}{8.42} &&\multirow{2}{*}{12.63}\\
                            &&&&&\\
\hline
                            &&\multirow{2}{*}{iter}            &\multirow{2}{*}{19\quad $+$\quad 7} &&\multirow{2}{*}{25} \\
                            &&&&&\\
\multirow{2}{*}{$2^{-7}$}   &\multirow{2}{*}{16129}&\multirow{2}{*}{residual $\eta$}
                            &\multirow{2}{*}{ 4.06e-11}        &   &\multirow{2}{*}{3.88e-11}          \\
                            &&&&&\\
                            &&\multirow{2}{*}{CPU time/s}          &\multirow{2}{*}{43.45} &&\multirow{2}{*}{65.18} \\
                            &&&&&\\
\hline
                            &&\multirow{2}{*}{iter}            &\multirow{2}{*}{20\quad $+$\quad 8} &&\multirow{2}{*}{25} \\
                            &&&&&\\
\multirow{2}{*}{$2^{-8}$}   &\multirow{2}{*}{65025}
                            &\multirow{2}{*}{residual $\eta$}
                            &\multirow{2}{*}{8.45e-12}        &     &\multirow{2}{*}{8.72e-12}             \\
                            &&&&&\\
                            &&\multirow{2}{*}{CPU time/s}          &\multirow{2}{*}{189.04} &&\multirow{2}{*}{283.20}\\
                            &&&&&\\
\hline
                            &&\multirow{2}{*}{iter}            &\multirow{2}{*}{20\quad $+$\quad 8} &&\multirow{2}{*}{26} \\
                            &&&&&\\
\multirow{2}{*}{$2^{-9}$}   &\multirow{2}{*}{261121}
                            &\multirow{2}{*}{residual $\eta$}
                            &\multirow{2}{*}{7.33e-12}       &     &\multirow{2}{*}{7.21e-12}            \\
                            &&&&&\\
                            &&\multirow{2}{*}{CPU time/s}&\multirow{2}{*}{2155.01} &&\multirow{2}{*}{3232.63}\\
                            &&&&&\\
\hline
                            &&\multirow{2}{*}{iter}            &\multirow{2}{*}{20\quad $+$\quad 8} &&\multirow{2}{*}{26} \\
                            &&&&&\\
\multirow{2}{*}{$2^{-10}$}   &\multirow{2}{*}{1046529}
                            &\multirow{2}{*}{residual $\eta$}
                            &\multirow{2}{*}{9.58e-12}       &     &\multirow{2}{*}{9.73e-12}            \\
                            &&&&&\\
                            &&\multirow{2}{*}{CPU time/s}&\multirow{2}{*}{58049.57} &&\multirow{2}{*}{87035.63}\\
                            &&&&&\\
\hline
\end{tabular}
\end{center}
\end{table}
\end{document}